\documentclass[10pt,reqno]{amsart}

\usepackage{packages}
\usepackage{theorems}
\usepackage{notation}

\title{
  Strict Complementarity in MaxCut SDP%
}

\author{Marcel K.\ de Carli Silva}
\address[Marcel K.\ de Carli Silva]{%
  Instituto de Matemática e Estatística, Universidade de São Paulo%
}
\email{mksilva@ime.usp.br}
\thanks{%
  Research of the first author was supported in part by Discovery
  Grants from NSERC and by U.S.~Office of Naval Research under award
  number N00014-12-1-0049, while at the Department of Combinatorics
  and Optimization, University of Waterloo, and by FAPESP
  (Proc.~2013/03447-6), CNPq (Proc.~477203/2012-4), CNPq
  (Proc.~456792/2014-7), and CAPES, while at the Institute of
  Mathematics and Statistics, University of São Paulo.%
}

\author{Levent Tunçel}
\address[Levent Tunçel]{%
  Department of Combinatorics and Optimization, University of Waterloo%
}
\email{ltuncel@uwaterloo.ca}
\thanks{%
  Research of the second author was supported in part by Discovery
  Grants from NSERC and by U.S.~Office of Naval Research under award
  numbers N00014-12-1-0049, N00014-15-1-2171, and~N00014-18-1-2078.
  Part of this work was done while the second author was visiting the
  Simons Institute for the Theory of Computing, supported in part by
  the DIMACS/Simons Collaboration on Bridging Continuous and Discrete
  Optimization through NSF grant \#CCF-1740425.%
}

\date{June 2, 2018}

\begin{document}

\begin{abstract}
  The MaxCut SDP is one of the most well-known semidefinite programs,
  and it has many favorable properties.  One of its nicest
  geometric/duality properties is the fact that the vertices of its
  feasible region correspond exactly to the cuts of a graph, as proved
  by Laurent and Poljak in~1995.  Recall that a boundary point~\(x\) of a
  convex set~\(\Cscr\) is called a vertex of~\(\Cscr\) if the normal
  cone of~\(\Cscr\) at~\(x\) is full-dimensional.

  We study how often strict complementarity holds or fails for the
  MaxCut SDP when a vertex of the feasible region is optimal, i.e.,
  when the SDP relaxation is tight.  While strict complementarity is
  known to hold when the objective function is in the interior of the
  normal cone at any vertex, we prove that it fails generically at the
  boundary of such normal cone.  In this regard, the MaxCut SDP
  displays the nastiest behavior possible for a convex optimization problem.

  We also study strict complementarity with respect to two classes of
  objective functions.  We~show~that, when the objective functions are
  sampled uniformly from the negative semidefinite rank-one matrices
  in the boundary of the normal cone at any vertex, the probability
  that strict complementarity holds lies in \((0,1)\).  We~also extend
  a construction due to Laurent and Poljak of weighted Laplacian
  matrices for which strict complementarity fails.  Their construction
  works for complete graphs, and we extend it to cosums of graphs under
  some mild conditions.
\end{abstract}

\maketitle

\section{Introduction}

Complementary slackness is a fundamental optimality condition, and
hence ubiquitous in optimization.  In~the most general setting of
nonlinear programming (formulated below in the classical language of
nonlinear optimization), it requires a pair \(\pdpair{\xb}{\yb}\) of
primal-dual feasible solutions to an optimization problem
\begin{equation}
  \label{opt-general}
  \min\setst[\big]{
    f(x)
  }{
    g_i(x) \geq 0\,
    \forall i \in \set{1,\dotsc,m}
  }
\end{equation}
and its (Lagrangean) dual to satisfy \(\yb_i g_i(\xb) = 0\) for every
\(i \in [m] \coloneqq \set{1,\dotsc,m}\); that~is, at least one of the
feasibility conditions \(g_i(\xb) \geq 0\) (in the primal) and \(\yb_i
\geq 0\) (in the dual) must be tight, i.e., they cannot both have a
slack.  In this case, we say that \(\pdpair{\xb}{\yb}\) is
\emph{complementary}.  This condition can be stated very conveniently
in structured convex optimization.  For a linear program (LP)
\begin{equation}
  \label{eq:generic-lp}
  \max\setst{
    \iprodt{c}{x}
  }{
    Ax = b,\,
    x \geq 0
  }
\end{equation}
and its dual
{\(
  \min\setst{
    \iprodt{b}{y}
  }{
    s = A^{\transp} y - c,\,
    s \geq 0
  }\),
}
where \(A \in \Reals^{m \times n}\) is a matrix, and \(b \in
\Reals^m\) and \(c \in \Reals^n\) are vectors, a pair
\(\pdpair{\xb}{\dualpair{\yb}{\sbar}}\) of primal-dual feasible
solutions is \emph{complementary} if \(\xb_i\sbar_i = 0\) for every
\(i \in [n]\).  One can similarly consider a semidefinite program
(SDP)
\begin{equation}
  \label{eq:generic-sdp}
  \max\setst{
    \Tr(CX)
  }{
    \Acal(X) = b,\,
    X \succeq 0
  };
\end{equation}
here, as usual, we equip the space~\(\Sym{n}\) of symmetric \(n\)-by-\(n\)
 matrices with the trace inner-product \(\iprod{C}{X} \coloneqq
\Tr(CX^{\transp}) = \sum_{i,j} C_{ij} X_{ij}\), the map \(\Acal \colon
\Sym{n} \to \Reals^m\) is linear, and \(X \succeq 0\) denotes that
\(X \in \Sym{n}\) is positive semidefinite; most of our notation can
be found
in~\Cref{tbl:special-sets,tbl:linear-algebra,tbl:convex,tbl:measure,tbl:graphs}.
The dual SDP is
{\(\min\setst{
    \iprodt{b}{y}
  }{
    S = \Acal^*(y) - C,\,
    S \succeq 0
  }\),
}
where \(\Acal^* \colon \Reals^m \to \Sym{n}\) is the adjoint
of~\(\Acal\), and a pair \(\pdpair{\Xb}{\dualpair{\yb}{\Sb}}\) of primal-dual
feasible solutions is called \emph{complementary} if
\(\Tr(\Xb\Sb) = 0\); equivalently, if \(\Xb\Sb = 0\), since
\(\Xb,\Sb \succeq 0\).

Strict complementarity is a refinement of the notion of complementary
slackness where we require \emph{precisely~one} of the feasibility
conditions involved to be tight, which forces the other one to have a
slack.  A pair \(\pdpair{\xb}{\yb}\) of primal-dual feasible solutions
for the optimization problem in~\eqref{opt-general} and its dual is
\emph{strictly complementary} if \(\yb_i g_i(\xb) = 0\) and \(\yb_i +
g_i(\xb) > 0\) for every \(i \in [m]\).  A pair
\(\pdpair{\xb}{\dualpair{\yb}{\sbar}}\) of primal-dual feasible
solutions for the LP in~\eqref{eq:generic-lp} and its dual is
\emph{strictly complementary} if \(\xb_i \sbar_i = 0\) and \(\xb_i +
\sbar_i > 0\) for every \(i \in [n]\).  Finally, a pair
\(\pdpair{\Xb}{\dualpair{\yb}{\Sb}}\) of primal-dual feasible
solutions for the SDP in~\eqref{eq:generic-sdp} and its dual is
\emph{strictly complementary} if \(\Xb\Sb = 0\) and \(\Xb + \Sb \succ
0\), i.e., \(\Xb + \Sb\) is positive definite.  The latter two notions
can be neatly unified in the context of convex conic optimization via
the concept of faces (see~\cite{Pataki00a}).

Complementary slackness characterizes optimality whenever Strong
Duality holds, in both LPs and SDPs: a~primal-dual pair of feasible
solutions is optimal if and only if it is complementary.  This is
sometimes described by saying that \emph{complementary slackness
  holds} for the (primal-dual pair of) programs.  In the case of LPs,
whenever primal and dual are both feasible, there exists a primal-dual
pair of optimal solutions that is strictly
complementary~\cite{GoldmanT56a}; i.e., \emph{strict complementarity
  holds} for every primal-dual pair of feasible~LPs.  However, there
exist primal-dual pairs of SDPs (which satisfy strong regularity
conditions sufficient for SDP Strong Duality) that have no strictly
complementary primal-dual pair of optimal solutions
(see~\cite{ShapiroS00a}); in~such~cases, we say that \emph{strict
  complementarity fails} for said primal-dual pair of SDPs.  In fact,
failure of strict complementarity is deeply related to failure of
Strong Duality in the context of convex conic
optimization~\cite{TuncelW12a}.

Existence of a strictly complementary pair of optimal solutions for an SDP is
crucial for some key properties of interior-point methods used to
solve such an optimization problem; see, e.g.,
\cite{AlizadehHO98a,JiPS99a,LuoSZ98a,KojimaSS98a} for superlinear convergence
and~\cite{HalickaKR02a} for convergence of the central path to the analytic center of the
optimal face.  Strict complementarity is also very useful in the
identification of optimal faces (in the primal and dual problems), for
detection of infeasibility and unboundedness as well as efficient
recovery of certificates of these \cite{YeTM94a,NesterovTY99a}. Hence,
it is important to determine whether strict complementarity holds for
a given SDP.

It is known that strict complementarity holds generically for
SDPs~\cite{AlizadehHO97a}; for a generalization to convex optimization
problems, see~\cite{PatakiT01a}.  However, there are some generic
properties of LPs that fail in some natural, highly structured
formulations arising in combinatorial optimization.  For instance,
whereas systems of linear inequalities are well-known to be
generically nondegenerate, the natural description of many classical
polytopes is degenerate (e.g., for the matching polytope,
see~\cite[Theorem~25.4]{Schrijver03a}), and ``\ldots most real-world
LP problems are degenerate'' according to~\cite{YeGTZ93a}.  Thus, one
ought to be careful about strict complementarity when approaching
combinatorial optimization problems via SDP relaxations.

In this paper, we study how often strict complementarity holds or
fails for the MaxCut SDP and its dual, when an optimal solution of the
primal occurs at a vertex of its feasible region.  Recall that the
\emph{MaxCut problem} for a given graph \(G = (V,E)\) on \(V = [n]\)
and weight function \(w \colon E \to \Reals\) can be cast as the
optimization problem
{\(
  \max\setst{
    \qform{C}{x}
  }{
    x \in \set{\pm1}^n
  }\),
}
where \(C \in \Sym{n}\) is defined as
\begin{equation}
  \label{eq:1}
  4C
  \coloneqq
  \Lcal_G(w)
  \coloneqq
  \sum_{\set{i,j} \in E} w_{\set{i,j}} \oprodsym{(e_i-e_j)}
\end{equation}
and \(\set{e_1,\dotsc,e_n}\) is the standard basis of~\(\Reals^n\).
The matrix \(\Lcal_G(w)\) is known as (a weighted) \emph{Laplacian}
matrix of~\(G\), and it is simple to check that \(\Lcal_G(w) \succeq
0\) if \(w \geq 0\).  The natural SDP relaxation for this problem is
the following \emph{MaxCut SDP}, which we write along with its dual:
\begin{equation}
  \label{eq:maxcut-sdp}
  \begin{array}[!h]{rl}
    \max & \Tr(CX)           \\
         & \diag(X) = \ones, \\
         & X \succeq 0,
  \end{array}
  \begin{array}[!h]{lrl}
    = & \min & \iprodt{\ones}{y} \\
      &      & S = \Diag(y) - C, \\
      &      & S \succeq 0;
  \end{array}
\end{equation}
here, \(\diag \colon \Sym{n} \to \Reals^n\) extracts the diagonal,
\(\Diag \colon \Reals^n \to \Sym{n}\) is the adjoint of~\(\diag\), and
\(\ones\) is the vector of all-ones.  Strong Duality holds for every
\(C \in \Sym{n}\) since both SDPs have \emph{Slater points}, i.e.,
feasible solutions that are positive definite.

The feasible region of the MaxCut SDP, called the \emph{elliptope} and
denoted by \(\Elliptope{n}\), is a compact convex set in~\(\Sym{n}\)
and its vertices are precisely its elements that are rank-one
matrices~\cite{LaurentP95a}, i.e., matrices of the form
\(\oprodsym{x}\) with \(x \in \set{\pm 1}^n\).  Thus, they correspond
precisely to the \emph{exact} solutions of the MaxCut problem,
for~which the SDP is a relaxation.  The vertices of~\(\Elliptope{n}\)
are by definition the points of~\(\Elliptope{n}\) whose normal cones
are full-dimensional (we postpone the definition of normal cone to
\Cref{sec:vertices}).  It~is known~\cite{CarliT15a} that strict
complementarity holds in~\eqref{eq:maxcut-sdp} precisely when \(C\)
lies in the relative interior of the normal cone of \emph{some} \(X
\in \Elliptope{n}\).  In~particular, if~\(\Xb\) is a vertex
of~\(\Elliptope{n}\), then strict complementarity holds
for~\eqref{eq:maxcut-sdp} whenever \(C\) is in the interior of the
normal cone of~\(\Elliptope{n}\) at~\(\Xb\).  However, when \(C\) lies
in the boundary of this normal cone, it is not clear whether strict
complementarity holds.

In this paper, we prove that, when \(C\) is chosen from the boundary
of the normal cone at a vertex of the elliptope, strict
complementarity almost always fails for~\eqref{eq:maxcut-sdp}; in this
regard, surprisingly, the MaxCut SDP displays the worst possible
behavior for a convex optimization problem.  In~order to make the statement ``almost always fails''
rigorous, we~shall make use of Hausdorff measures.  However, our
treatment is self-contained and it does not require in-depth knowledge of the theory of
Hausdorff measures.

We also focus on two classes of objective functions
for~\eqref{eq:maxcut-sdp}.  We prove that, when \(C\) is sampled
uniformly from (a normalization of) the negative semidefinite rank-one
matrices in the normal cone at a vertex of the elliptope, the
probability that strict complementarity fails
for~\eqref{eq:maxcut-sdp} is in~\((0,1)\).  Naturally, we shall also
use Hausdorff measures to achieve this.  Finally, we also extend a
construction due to Laurent and Poljak~\cite{LaurentP96a}, who proved
that strict complementarity may fail for~\eqref{eq:maxcut-sdp} when
\(C\) is a weighted Laplacian matrix.  Their construction works for
complete graphs, and we extend it to graphs which are cosums where one
of the summands is connected and with some mild condition relating the
maximum eigenvalues of their Laplacians.

The order in which our results are presented is different from what we
described above.  Since the weighted Laplacian construction
generalized from Laurent and Poljak involves only matrix analysis and
spectral graph theory, and no measure theory, we start with that
result.  Only then we shall delve into measure theory tools to prove
the other results.  Hence, the rest of this paper is organized as
follows.  \Cref{sec:prelim} contains some preliminaries, such as
notation and background results about the MaxCut
SDP~\eqref{eq:maxcut-sdp}.  In \Cref{sec:maxcut-sc-failure} we discuss
failure of strict complementarity for~\eqref{eq:maxcut-sdp} using
previous results by Laurent and Poljak and we extend their Laplacian
construction to cosums of graphs.  In~\Cref{sec:generic-sc-failure}, we
develop some Hausdorff measure basics and use them to prove that
strict complementarity fails generically (``almost everywhere'') for
the MaxCut~SDP~\eqref{eq:maxcut-sdp} when the objective function is in
the boundary of the normal cone of a vertex of the~elliptope.
Finally, in \Cref{sec:rank-one2}, we zoom into the set of negative
semidefinite rank-one matrices in the latter boundary, and prove that
in this case the probability that strict complementarity holds is
in~\((0,1)\).

\section{Preliminaries}
\label{sec:prelim}

We refer the reader to
\Cref{tbl:special-sets,tbl:linear-algebra,tbl:convex,tbl:measure,tbl:graphs} for
our mostly standard notation and terminology.  In order to treat
\(\Reals^n\) and \(\Sym{n}\) uniformly, we adopt the language of
Euclidean spaces, i.e., finite-dimensional real vectors spaces
equipped with an inner product.  We denote arbitrary Euclidean spaces
by~\(\Ebb\) and~\(\Ybb\).  We adopt Minkowski's notation; for
instance,
{\(
  \Cscr + \Lambda \Dscr
  \coloneqq
  \setst{
    x + \lambda y
  }{
    x \in \Cscr,\,
    \lambda \in \Lambda,\,
    y \in \Dscr
  }\)
}
for \(\Cscr,\Dscr \subseteq \Ebb\) and \(\Lambda \subseteq \Reals\).
Also, whenever possible we shorten singletons to their single
elements, e.g., we write \(\Reals_+(1 \oplus \Cscr)\) to denote the
conic homogenization of the set \(\Cscr\) in one higher dimensional space.

\bgroup
\renewcommand{\arraystretch}{1.2}
\begin{table}[!ht]
  \caption{Notation for special sets.}
  \centering
  \begin{tabular}{r c p{12cm} }
    \toprule
    \([n]\)
    & \(\coloneqq\) &
    \(\set{1,\dotsc,n}\) for each \(n \in \Naturals\)
    \\
    \(\Powerset{X}\)
    & \(\coloneqq\) &
    the power set of~\(X\)
    \\
    \(\Reals_+\)
    & \(\coloneqq\) &
    \(\setst{x \in \Reals}{x \geq 0}\), the set of nonnegative reals
    \\
    \(\Reals_{++}\)
    & \(\coloneqq\) &
    \(\setst{x \in \Reals}{x > 0}\), the set of positive reals
    \\
    \(\Reals^{n \times n}\)
    & \(\coloneqq\) &
    the space of \(n \times n\) real-valued matrices
    \\
    \(\Sym{n}\)
    & \(\coloneqq\) &
    \(\setst{X \in \Reals^{n \times n}}{X = X^{\transp}}\),
    the space of symmetric \(n \times n\) matrices
    \\
    \(\Psd{n}\)
    & \(\coloneqq\) &
    \(\setst{X \in \Sym{n}}{\qform{X}{h} \geq 0\,\forall h \in \Reals^n}\),
    the cone of \emph{positive semidefinite} matrices
    \\
    \(\Pd{n}\)
    & \(\coloneqq\) &
    \(\setst{X \in \Sym{n}}{\qform{X}{h} > 0\,\forall h \in \Reals^n \setminus \set{0}}\),
    the cone of \emph{positive definite} matrices
    \\
    \(\Elliptope{n}\)
    & \(\coloneqq\) &
    the \emph{elliptope}; see~\eqref{eq:4}
    \\
    \bottomrule
  \end{tabular}
  \label{tbl:special-sets}
\end{table}
\egroup                         

\bgroup
\renewcommand{\arraystretch}{1.2}
\begin{table}[!ht]
  \caption{Notation for linear algebra.}
  \centering
  \begin{tabular}{r c p{12cm} }
    \toprule
    \(\Acal^*\)
    & \(\coloneqq\) &
    the \emph{adjoint} of a linear map~\(\Acal\) between Euclidean spaces
    \\
    \(\Tr(X)\)
    & \(\coloneqq\) &
    \(\sum_{i=1}^n X_{ii}\), the \emph{trace} of \(X \in \Reals^{n \times n}\)
    \\
    \(I\)
    & \(\coloneqq\) &
    the identity matrix in the appropriate space
    \\
    \(\ones\)
    & \(\coloneqq\) &
    the vector of all-ones in the appropriate space
    \\
    \(\set{e_1,\dotsc,e_n}\)
    & \(\coloneqq\) &
    the set of canonical basis vectors of~\(\Reals^n\)
    \\
    \(\Image(A)\)
    & \(\coloneqq\) &
    the range of \(A \in \Reals^{n \times n}\)
    \\
    \(\Null(A)\)
    & \(\coloneqq\) &
    the nullspace of \(A \in \Reals^{n \times n}\)
    \\
    \(\supp(x)\)
    & \(\coloneqq\) &
    \(\setst{i \in [n]}{x_i \neq 0}\),
    the \emph{support} of \(x \in \Reals^n\)
    \\
    \(\diag(X)\)
    & \(\coloneqq\) &
    \(\sum_{i=1}^n X_{ii}e_i\) for each \(X \in \Reals^{n \times n}\)
    so \(\diag \colon \Reals^{n \times n} \to \Reals^n\) extracts the diagonal
    \\
    \(\Diag(x)\)
    & \(\coloneqq\) &
    \(\sum_{i=1}^n x_i \oprodsym{e_i} \in \Reals^{n \times n}\)
    for each \(x \in \Reals^n\),
    so \(\Diag\) is the adjoint of \(\diag\)
    \\
    \(\Cscr^{\perp}\)
    & \(\coloneqq\) &
    \(\setst{x \in \Ebb}{\iprod{x}{s} = 0\,\forall s \in \Cscr}\)
    for each subset \(\Cscr\) of an Euclidean space~\(\Ebb\)
    \\
    \(\oplus\)
    & \(\coloneqq\) &
    the direct sum of two vectors or two sets of vectors
    \\
    \(x \perp y\)
    & \(\coloneqq\) &
    denotes that \(x,y \in \Ebb\) are orthogonal, i.e., \(\iprod{x}{y} = 0\)
    \\
    \(\succeq\)
    & \(\coloneqq\) &
    the \emph{Löwner partial order} on \(\Sym{n}\), i.e., \(A \succeq B \iff A-B \in \Psd{n}\) for \(A,B \in \Sym{n}\)
    \\
    \(\succ\)
    & \(\coloneqq\) &
    the partial order on~\(\Sym{n}\) defined as \(A \succ B \iff A-B \in \Pd{n}\) for \(A,B \in \Sym{n}\)
    \\
    \(\lambda_{\max}(A)\)
    & \(\coloneqq\) &
    the largest eigenvalue of \(A \in \Sym{n}\)
    \\
    \(A^{\MPinv}\)
    & \(\coloneqq\) &
    the Moore-Penrose pseudoinverse of \(A \in \Reals^{m \times n}\);
    see~\cite{HornJ90a}
    \\
    \(\matvec\)
    & \(\coloneqq\) &
    the map that sends a matrix in \(\Reals^{n \times n}\) to a vector
    indexed by \([n] \times [n]\)
    \\
    \bottomrule
  \end{tabular}
  \label{tbl:linear-algebra}
\end{table}
\egroup                         

\bgroup
\renewcommand{\arraystretch}{1.2}
\begin{table}[!ht]
  \caption{Notation for convex analysis on an Euclidean space~\(\Ebb\).}
  \centering
  \begin{tabular}{r c p{12cm} }
    \toprule
    \(\cl(\Cscr)\)
    & \(\coloneqq\) &
    the \emph{closure} of \(\Cscr \subseteq \Ebb\)
    \\
    \(\interior(\Cscr)\)
    & \(\coloneqq\) &
    the \emph{interior} of \(\Cscr \subseteq \Ebb\)
    \\
    \(\ri(\Cscr)\)
    & \(\coloneqq\) &
    the \emph{relative interior} of a convex set \(\Cscr \subseteq \Ebb\)
    \\
    \(\bd(\Cscr)\)
    & \(\coloneqq\) &
    \(\cl(\Cscr) \setminus \interior(\Cscr)\), the \emph{boundary} of \(\Cscr \subseteq \Ebb\)
    \\
    \(\rbd(\Cscr)\)
    & \(\coloneqq\) &
    \(\cl(\Cscr) \setminus \ri(\Cscr)\),
    the \emph{relative boundary} of a convex set \(\Cscr \subseteq \Ebb\)
    \\
    \(\Fscr \faceeq \Cscr\)
    & \(\coloneqq\) &
    denotes that \(\Fscr\) is a face of a convex set \(\Cscr \subseteq \Ebb\); see \Cref{sec:bdstruct}
    \\
    \(\Fscr \faceneq \Cscr\)
    & \(\coloneqq\) &
    denotes that \(\Fscr\) is a proper face of a convex set \(\Cscr \subseteq \Ebb\); see \Cref{sec:bdstruct}
    \\
    \(\Faces(\Cscr)\)
    & \(\coloneqq\) &
    the set of faces of a convex set \(\Cscr \subseteq \Ebb\); see \Cref{sec:bdstruct}
    \\
    \(\Normal{\Cscr}{x}\)
    & \(\coloneqq\) &
    the normal cone of a convex set \(\Cscr \subseteq \Ebb\) at \(x \in \Cscr\); see \eqref{eq:3}
    \\
    \(\Ball\)
    & \(\coloneqq\) &
    the unit ball in the appropriate Euclidean space
    \\
    \(\Ball_{\infty}\)
    & \(\coloneqq\) &
    the unit ball for the \(\infty\)-norm in the appropriate \(\Reals^n\)
    \\
    \bottomrule
  \end{tabular}
  \label{tbl:convex}
\end{table}
\egroup                         

\bgroup
\renewcommand{\arraystretch}{1.2}
\begin{table}[!ht]
  \caption{Notation for the theory of Hausdorff measures in a normed space~\(\Vscr\).}
  \centering
  \begin{tabular}{r c p{12cm} }
    \toprule
    \(H_d(\Xscr)\)
    & \(\coloneqq\) &
    the \(d\)-dimensional Hausdorff outer measure of \(\Xscr \subseteq \Vscr\); see \eqref{eq:15}
    \\
    \(\lambda_d(\Xscr)\)
    & \(\coloneqq\) &
    the \(d\)-dimensional Lebesgue outer measure of \(\Xscr \subseteq \Reals^d\)
    \\
    \(\dim_H(\Xscr)\)
    & \(\coloneqq\) &
    the Hausdorff dimension of \(\Xscr \subseteq \Vscr\); see~\eqref{eq:17}
    \\
    \bottomrule
  \end{tabular}
  \label{tbl:measure}
\end{table}
\egroup                         

\bgroup
\renewcommand{\arraystretch}{1.2}
\begin{table}[!ht]
  \caption{Notation for a graph \(G = (V,E)\).}
  \centering
  \begin{tabular}{r c p{12cm} }
    \toprule
    \(V(G)\)
    & \(\coloneqq\) &
    the vertex set of~\(G\)
    \\
    \(E(G)\)
    & \(\coloneqq\) &
    the edge set of~\(G\)
    \\
    \(\Lcal_G(w)\)
    & \(\coloneqq\) &
    the weighted Laplacian matrix of~\(G\) with weights \(w \in \Reals^E\); see~\eqref{eq:1}
    \\
    \(G \cosum H\)
    & \(\coloneqq\) &
    the cosum of graphs \(G\) and~\(H\); see~\eqref{eq:6}
    \\
    \bottomrule
  \end{tabular}
  \label{tbl:graphs}
\end{table}
\egroup                         

\subsection{Uniqueness of Dual Optimal Solutions}

Delorme and Poljak~\cite{DelormeP93b} proved that the dual SDP
in~\eqref{eq:maxcut-sdp} has a unique optimal solution.  We shall
state a slightly generalized version of their result with some changes
and include a proof for the sake of completeness.
\begin{proposition}[{\cite[Theorem~2]{DelormeP93b}}]
  \label{prop:unique}
  Consider the primal-dual pair of SDPs
  \begin{align}
    \label{eq:unique-primal}
    & \max\setst{
      \Tr(CX)
    }{
      \Acal(X) = b,\,
      X \succeq 0
    }
    \qquad\text{and}
    \\
    \label{eq:unique-dual}
    & \min\setst{
      \iprodt{b}{y}
    }{
      S = \Acal^*(y) - C,\,
      S \succeq 0
    },
  \end{align}
  where \(\Acal \colon \Sym{n} \to \Reals^m\) is a surjective linear
  map, \(C \in \Sym{n}\), and \(b \in \Reals^m\).  Assume there exist
  \(\Xcirc \in \Pd{n}\) and \(\ycirc \in \Reals^m\) such that
  \(\Acal(\Xcirc) = b\) and \(\Acal^*(\ycirc) \in \Pd{n}\).  Suppose
  that, for every nonzero \(y \in \Reals^m\), there exists \(z \in
  \Reals^m\) such that \(\iprodt{b}{z} \neq 0\) and
  \(\Null(\Acal^*(y)) \subseteq \Null(\Acal^*(z))\).  Then
  \eqref{eq:unique-dual} has a unique optimal solution.
\end{proposition}

\begin{proof}
  Since \(\Xcirc\) is a Slater point for~\eqref{eq:unique-primal},
  there exists an optimal solution for~\eqref{eq:unique-dual}.
  Suppose for the sake of contradiction that \(\dualpair{y_1}{S_1}\)
  and \(\dualpair{y_2}{S_2}\) are distinct optimal solutions
  for~\eqref{eq:unique-dual}.  Set \(\yb \coloneqq
  \tfrac{1}{2}(y_1+y_2)\) and \(\Sb \coloneqq \Acal^*(\yb) - C =
  \tfrac{1}{2}(S_1+S_2) \succeq 0\).  We have \(\Sb \neq 0\) since \(\Acal\) is
  surjective.  Then \(\dualpair{\yb}{\Sb}\) is also optimal
  in~\eqref{eq:unique-dual}.  Let \(\zb \in \Reals^m\) such that
  \(\iprodt{b}{\zb} \neq 0\) and \(\Null(\Acal^*(y_1-y_2)) \subseteq
  \Null(\Acal^*(\zb))\), which exists by assumption.  Then
  \begin{equation}
    \label{eq:2}
    \Null(\Sb) \subseteq \Null(\Acal^*(\zb));
  \end{equation}
  indeed, if \(h\) lies in \(\Null(\Sb) = \Null(S_1) \cap
  \Null(S_2)\), then we get \(\Acal^*(y_1) h = C h = \Acal^*(y_2) h\),
  whence \(h \in \Null(\Acal^*(y_1-y_2)) \subseteq
  \Null(\Acal^*(\zb))\).

  Define
  \begin{equation*}
    \beta
    \coloneqq
    -
    \dfrac{
      \iprodt{b}{\ycirc}
    }{
      \iprodt{b}{\zb}
    },
    \qquad
    d
    \coloneqq
    \ycirc + \beta \zb,
  \end{equation*}
  and note that \(\iprodt{b}{d} = 0\).  Let \(\mu > 0\) be the
  smallest positive eigenvalue of \(\Sb \in \Psd{n} \setminus
  \set{0}\).  Let \(\norm{\cdot}_2\) denote the operator \(2\)-norm.
  If \(\beta\norm{\Acal^*(\zb)}_2 = 0\), set \(\eps \coloneqq 1\);
  otherwise set
  \begin{gather*}
    \eps \coloneqq \frac{\mu}{\abs{\beta}\norm{\Acal^*(\zb)}_2} > 0.
  \end{gather*}
  Also, set \(\yt \coloneqq \yb + \eps d\) and \(\St \coloneqq
  \Acal^*(\yt) - C\).  Let \(h \in \Reals^n\).  Write \(h = h_1 +
  h_2\) with \(h_1 \in \Null(\Sb)\) and \(h_2 \in
  [\Null(\Sb)]^{\perp}\).  By~\eqref{eq:2} we have
  \begin{equation*}
    \begin{split}
      \qform{\St}{h}
      & =
      \qform{\Sb}{h}
      +
      \eps\qform{\Acal^*(d)}{h}
      \\
      & \geq
      \mu \norm{h_2}^2
      +
      \eps\qform{\Acal^*(\ycirc)}{h}
      +
      \eps\beta\qform{\Acal^*(\zb)}{h}
      \\
      & \geq
      \mu \norm{h_2}^2
      +
      \eps\qform{\Acal^*(\ycirc)}{h}
      -
      \eps\abs{\beta}\norm{\Acal^*(\zb)}_2\norm{h_2}^2
      \\
      & \geq
      \eps\qform{\Acal^*(\ycirc)}{h}.
    \end{split}
  \end{equation*}
  Thus, \(\St \succeq \eps\Acal^*(\ycirc) \succ 0\), so there exists a
  feasible solution for~\eqref{eq:unique-dual} with objective value
  strictly smaller than \(\iprodt{b}{\yt} = \iprodt{b}{\yb}\), a
  contradiction.
\end{proof}

\begin{corollary}[{\cite[Theorem~2]{DelormeP93b}}]
  \label{cor:unique-maxcut}
  The dual SDP in~\eqref{eq:maxcut-sdp} has a unique optimal solution.
\end{corollary}
\begin{proof}
  We shall apply Proposition~\ref{prop:unique}
  to~\eqref{eq:maxcut-sdp}.  Let us see that the map \(\Acal \coloneqq
  \diag\) satisfies the required properties.  Take \(\Xcirc \coloneqq
  I\) and \(\ycirc \coloneqq \ones\).  Let \(y \in \Reals^n\) be
  nonzero.  Define \(z \in \Reals^n\) as \(z_i \coloneqq \abs{y_i}\)
  for every \(i \in [n]\), and note that \(\Null(\Diag(y)) =
  \Null(\Diag(z))\) and that \(\iprodt{\ones}{z} > 0\) since \(y \neq
  0\).
\end{proof}

\subsection{Vertices of the Elliptope}
\label{sec:vertices}

Let \(\Cscr\) be a convex set in an Euclidean space~\(\Ebb\).  The
\emph{normal cone} of~\(\Cscr\) at \(\xb \in \Cscr\) is
\begin{equation}
  \label{eq:3}
  \Normal{\Cscr}{\xb}
  \coloneqq
  \setst{
    a \in \Ebb
  }{
    \iprod{a}{x} \leq \iprod{a}{\xb}\,
    \forall x \in \Cscr
  },
\end{equation}
i.e., it is the set of all normals to supporting halfspaces
of~\(\Cscr\) at~\(\xb\).  Note that we are identifying the dual
space~\(\Ebb^*\) of~\(\Ebb\) with~\(\Ebb\).  We say that \(\xb \in
\Cscr\) is a \emph{vertex} of~\(\Cscr\) if \(\Normal{\Cscr}{\xb}\) is
full-dimensional.  The set of vertices of the \emph{elliptope}
\begin{equation}
  \label{eq:4}
  \Elliptope{n}
  \coloneqq
  \setst{X \in \Psd{n}}{\diag(X) = \ones}
\end{equation}
was determined by Laurent and Poljak~\cite{LaurentP95a}:
\begin{theorem}[{\cite[Theorem~2.5]{LaurentP95a}}]
  \label{thm:vertices-elliptope}
  The set vertices of~\(\Elliptope{n}\) is
  \(\setst[\big]{\oprodsym{x}}{x \in \set{\pm1}^n}\).
\end{theorem}

An \emph{automorphism} of~\(\Elliptope{n}\) is a nonsingular linear
operator~\(\Tcal\) on~\(\Sym{n}\) that preserves~\(\Elliptope{n}\),
i.e., \(\Tcal(\Elliptope{n}) = \Elliptope{n}\).  For \(s \in
\set{\pm1}^n\), the map \(X \in \Sym{n} \mapsto \Diag(s) X \Diag(s)\)
is easily checked to be an automorphism of~\(\Elliptope{n}\).  If
\(x,y \in \set{\pm1}^n\), then \(y = \Diag(s) x\) for \(s \in
\set{\pm1}^n\) defined by \(s_i \coloneqq x_i y_i\) for each \(i \in
[n]\).  Hence, any vertex of~\(\Elliptope{n}\) can be mapped into the
vertex \(\oprodsym{\ones}\) by an automorphism of~\(\Elliptope{n}\);
i.e., the automorphism group of~\(\Elliptope{n}\) acts transitively on
the vertices of~\(\Elliptope{n}\).  This allows us to prove many
linear properties about the vertices of~\(\Elliptope{n}\) by just
proving them for the vertex~\(\oprodsym{\ones}\).  \emph{We shall make
  extensive use of this fact without further mention.}

Laurent and Poljak~\cite{LaurentP95a} also provided formulas for the
normal cones of the elliptope.  Here we shall use slightly different
formulas from~\cite[Proposition~2.1]{CarliT15a}:
\begin{equation}
  \label{eq:normal}
  \begin{split}
    \Normal{\Elliptope{n}}{X}
    & =
    \Image(\Diag) - (\Psd{n} \cap \set{X}^{\perp})
    \\
    & =
    \Image(\Diag) - \setst{Y \in \Psd{n}}{\Image(Y) \subseteq \Null(X)}
  \end{split}
  \qquad
  \forall X \in \Elliptope{n}.
\end{equation}
When \(\Xb\) is a vertex of~\(\Elliptope{n}\), every element of
\(\Normal{\Elliptope{n}}{\Xb}\) can be described in a unique way as an
element of the Minkowski sum at the RHS of~\eqref{eq:normal}:
\begin{lemma}
  \label{le:unique}
  Let \(\Xb\) be a vertex of~\(\Elliptope{n}\).  Let \(\yb,\yt \in
  \Reals^n\) and \(\Sb,\St \in \Psd{n} \cap \set{\Xb}^{\perp}\) be
  such that \(\Diag(\yb) + \Sb = \Diag(\yt) + \St\).  Then \(\yb =
  \yt\) and \(\Sb = \St\).
\end{lemma}
\begin{proof}
  We may assume that \(\Xb = \oprodsym{\ones}\).  Then \(\Sb \in
  \Psd{n} \cap \set{\oprodsym{\ones}}^{\perp}\) implies that \(\Sb
  \ones = 0\).  Analogously, \(\St \ones = 0\).  Thus \(\yb =
  \Diag(\yb)\ones = (\Diag(\yb)+\Sb)\ones = (\Diag(\yt)+\St)\ones =
  \Diag(\yt)\ones = \yt\), so \(\Sb = \St\).
\end{proof}

\section{Failure of Strict Complementarity with Laplacian Objectives}
\label{sec:maxcut-sc-failure}

Existence of strictly complementary optimal solutions is known to be
equivalent to membership of the objective vector in the relative
interior of some normal cone:
\begin{proposition}[{\cite[Proposition~4.2]{CarliT15a}}]
  \label{prop:sc-ri}
  If the feasible region \(\Cscr\) of the SDP~\eqref{eq:generic-sdp} has a
  positive definite matrix, then strict complementarity holds
  for~\eqref{eq:generic-sdp} and its dual if and only if \(C \in
  \ri(\Normal{\Cscr}{X})\) for some \(X \in \Cscr\).
\end{proposition}

Hence, strict complementarity is locally generic when the objective
function is chosen in the normal cone of a given feasible solution;
see~\cite[Corollary~4.3]{CarliT15a}.

By~\eqref{eq:normal} and standard convex analysis,
\begin{equation}
  \label{eq:5}
  \begin{split}
    \ri(\Normal{\Elliptope{n}}{X})
    & =
    \Image(\Diag) - \ri(\Psd{n} \cap \set{X}^{\perp})
    \\
    & =
    \Image(\Diag) - \setst{Y \in \Psd{n}}{\Image(Y) = \Null(X)}
  \end{split}
  \qquad
  \forall X \in \Elliptope{n}.
\end{equation}
When \(\Xb\) is a vertex of~\(\Elliptope{n}\), we may
combine~\eqref{eq:normal} with \eqref{eq:5} and Lemma~\ref{le:unique}
to conclude that
\begin{equation}
  \label{eq:bd-normal-vx}
  \begin{split}
    \bd(\Normal{\Elliptope{n}}{\Xb})
    & =
    \Image(\Diag) - \rbd(\Psd{n} \cap \set{\Xb}^{\perp})
    \\
    & =
    \Image(\Diag) - \setst{Y \in \Psd{n}}{\Image(Y) \subsetneq \Null(\Xb)}.
  \end{split}
\end{equation}

In~\cite{CarliT15a}, we noted that strict complementarity holds
for~\eqref{eq:maxcut-sdp} for every \(C\) of the form \(C =
\tfrac{1}{4} \Lcal_G(w)\) with \(w \geq 0\) provided that the
polar~\(\Elliptope{n}^{\circ} \coloneqq \setst{Y \in \Sym{n}}{\Tr(YX)
  \leq 1 \,\forall X \in \Elliptope{n}}\) of the elliptope is facially
exposed, and we (implicitly) asked whether the latter holds.  It turns
out, Laurent and Poljak~\cite[Example~5.10]{LaurentP96a} showed, even
before we raised the question, in a different context and using a
slightly different terminology, that strict complementarity may fail
for~\eqref{eq:maxcut-sdp} for every \(n \geq 3\), hence answering the
question in the negative.  For each complete graph \(G = K_n\) with
\(n \geq 3\), they provided a weight function \(w \geq 0\) for which
strict complementarity fails for~\eqref{eq:maxcut-sdp} with \(C =
\tfrac{1}{4} \Lcal_G(w)\).

We generalize their construction showing that strict complementarity
may fail with a weighted Laplacian objective for graphs which are
cosums, with mild conditions on the (co-)summands.  Recall that, if
\(G = (V,E)\) and \(H = (U,F)\) are graphs such that \(V \cap U =
\emptyset\), the \emph{cosum} of~\(G\) and~\(H\) is the graph
\begin{equation}
  \label{eq:6}
  G \cosum H
  \coloneqq
  (V \cup U, E \cup F \cup \setst{\set{v,u}}{(v,u) \in V \times U}).
\end{equation}
We shall use a characterization of positive semidefinite matrices
partitioned in blocks using Schur complements and the Moore-Penrose
pseudoinverse:
\begin{lemma}[{see~\cite[Theorem~4.3]{Gallier10a}}]
  For \(A \in \Sym{m}\), \(C \in \Sym{n}\), and \(B \in \Reals^{m
    \times n}\), we have
  \begin{equation}
    \label{eq:7}
    \begin{bmatrix}
      A           & B \\
      B^{\transp} & C \\
    \end{bmatrix}
    \succeq 0
    \iff
    A \succeq 0,
    \quad
    (I-AA^{\MPinv})B = 0,
    \quad
    \text{and }
    C \succeq B^{\transp}A^{\MPinv}B.
  \end{equation}
\end{lemma}

\begin{theorem}
  \label{thm:cosum}
  Let \(G\) and~\(H\) be graphs with \(n_G \geq 1\) and \(n_H \geq 1\)
  vertices, respectively.  Let \(w_G \colon E(G) \to \Reals_{++}\) and
  \(w_H \colon E(H) \to \Reals_{++}\) be weight functions, and denote
  the respective weighted Laplacians by \(L_G \coloneqq \Lcal_G(w_G)\)
  and \(L_H \coloneqq \Lcal_H(w_H)\).  Set \(\mu_G \coloneqq
  \lambda_{\max}(L_G)\) and \(\mu_H \coloneqq \lambda_{\max}(L_H)\).
  Suppose that \(n_G \mu_G > n_H \mu_H\) and that \(H\) is connected.
  Define \(\wb \colon E(G \cosum H) \to \Reals_{++}\) as \(\wb
  \coloneqq w_G \oplus w_H \oplus \alpha \ones\) where \(\alpha
  \coloneqq \mu_G/n_H\).  For enhanced clarity denote the vectors of
  all-ones in \(\Reals^{V(G)}\) and~\(\Reals^{V(H)}\) by \(\ones_G\)
  and \(\ones_H\), respectively.  Then the unique pair of primal-dual
  optimal solutions for~\eqref{eq:maxcut-sdp} with \(4C \coloneqq
  \Lcal_{G \cosum H}(\wb)\) is \(\pdpair{X^*}{\dualpair{y^*}{S^*}}\)
  where
  \begin{equation}
    X^* \coloneqq
    \oprodsym{
      \begin{bmatrix*}[r]
        -\ones_G\, \\
        \ones_H\,  \\
      \end{bmatrix*}
    }
    \quad
    y^* \coloneqq 2\alpha
    \begin{bmatrix*}[r]
      \,n_H \ones_G\, \\
      \,n_G \ones_H\, \\
    \end{bmatrix*},
    \quad
    S^* \coloneqq
    \begin{bmatrix}
      \mu_G I - L_G                   & \alpha \oprod{\ones_G}{\ones_H} \\
      \alpha \oprod{\ones_H}{\ones_G} & \alpha n_G I - L_H              \\
    \end{bmatrix}.
  \end{equation}
  In particular, since \((X^*+S^*)(h \oplus 0) = 0\) for any
  \(\mu_G\)-eigenvector~\(h\) of~\(L_G\), there is no strictly
  complementary pair of primal-dual optimal solutions
  for~\eqref{eq:maxcut-sdp}.
\end{theorem}

\begin{proof}
  It is easy to check that \(X^*\) is feasible in the primal.  We have
  \[
  S^*
  =
  2
  \begin{bmatrix}
    \mu_G I & 0^{\transp}\, \\
    0       & \alpha n_G I  \\
  \end{bmatrix}
  -
  \begin{bmatrix}
    L_G + \mu_G I                   & -\alpha\oprod{\ones_G}{\ones_H} \\
    -\alpha\oprod{\ones_H}{\ones_H} & L_H + \alpha n_G I              \\
  \end{bmatrix}
  =
  \Diag(y^*) - L_{G \cosum H}(\wb),
  \]
  and the condition \(S^* \succeq 0\) is equivalent to the conditions
  \begin{subequations}
    \label{eq:8}
    \begin{gather}
      \label{eq:9}
      A \coloneqq \mu_G I - L_G \succeq 0,
      \\
      \label{eq:10}
      (I-AA^{\MPinv})\ones_G = 0,
      \\
      \label{eq:11}
      \alpha n_G I \succeq L_H + \alpha^2 \qform{A^{\MPinv}}{\ones_G} \oprodsym{\ones_H}.
    \end{gather}
  \end{subequations}
  Note that~\eqref{eq:9} holds trivially.  Also \(A\ones_G =
  \mu_G\ones_G\), so \(\ones_G \in \Image(A)\) and \eqref{eq:10} holds
  since \(I-AA^{\MPinv}\) is the orthogonal projector onto \(\Null(A)
  = \Image(A)^{\perp}\).  Finally, \(A^{\MPinv}\ones_G =
  \mu_G^{-1}\ones_G\) so \eqref{eq:11} is equivalent to \(\alpha n_G I
  \succeq L_H + \alpha n_G \frac{1}{n_H} \oprodsym{\ones_H}\), which
  holds since \(\alpha n_G > \mu_H\) by assumption.  It follows that
  \(\dualpair{y^*}{S^*}\) is feasible in the dual.  It is easy to
  check that \(\Tr(X^*S^*) = 0\), so~\(X^*\) and
  \(\dualpair{y^*}{S^*}\) are optimal solutions.  By
  Corollary~\ref{cor:unique-maxcut}, \(\dualpair{y^*}{S^*}\) is the
  unique optimal solution for the dual.

  It remains to show that \(X^*\) is the unique optimal solution for
  the primal.  Let
  \begin{equation*}
    X =
    \begin{bmatrix}
      X_G         & B   \\
      B^{\transp} & X_H \\
    \end{bmatrix}
  \end{equation*}
  be an optimal solution for the primal.  Complementary slackness
  yields
  \begin{equation}
    \label{eq:12}
    0
    =
    X S^*
    =
    \begin{bmatrix}
      X_G(\mu_{G}I-L_G)+{\alpha}B\oprod{\ones_H}{\ones_G}           & {\alpha}X_G\oprod{\ones_G}{\ones_H}+B({\alpha}n_{G}I-L_H)           \\
      B^{\transp}(\mu_{G}I-L_G)+{\alpha}X_H\oprod{\ones_H}{\ones_G} & {\alpha}B^{\transp}\oprod{\ones_G}{\ones_H}+X_H({\alpha}n_{G}I-L_H) \\
    \end{bmatrix}.
  \end{equation}
  If \(h \perp \ones_H\) is an eigenvector of~\(L_H\),
  (left-)multiplying~\(h\) by the bottom right block in~\eqref{eq:12} yields
  \(X_{H}h = 0\), where we used the assumption that \({\alpha}n_G >
  \mu_H\).  Since \(H\) is connected, this implies that \(X_H\) is a
  nonnegative scalar multiple of~\(\oprodsym{\ones_H}\), and so
  \begin{equation*}
    X_H = \oprodsym{\ones_H}.
  \end{equation*}
  Next apply \(\iprodt{\ones_G}{\cdot\ones_H}\) and
  \(\iprodt{\ones_H}{\cdot\ones_G}\) to the top right block and bottom
  left block of~\eqref{eq:12}, respectively, to get
  \begin{gather}
    \label{eq:13}
    0 = n_H\qform{X_G}{\ones_G} + n_G\iprodt{\ones_G}{B\ones_H},
    \\
    \label{eq:14}
    0 = n_H\iprodt{\ones_H}{B^{\transp}\ones_G} + n_G\qform{X_H}{\ones_H}.
  \end{gather}
  Hence,
  \begin{equation*}
    \frac{
      \qform{X_G}{\ones_G}
    }{
      n_G^2
    }
    =
    \frac{
      \qform{X_H}{\ones_H}
    }{
      n_H^2
    }
    \qquad\text{and}\qquad
    X_G = \oprodsym{\ones_G}.
  \end{equation*}
  Finally, by~\eqref{eq:13} we get \(\iprodt{\ones_G}{B\ones_H} =
  -n_{G}n_{H}\), and so \(B = -\oprod{\ones_G}{\ones_H}\).  Hence, \(X = X^*\).
\end{proof}

Note that the dimension of the \(\lambda_{\max}(\Lcal_G(w))\)-eigenspace
controls the ``degree'' to which strict complementarity fails
in~\Cref{thm:cosum}.  In~particular, when \(G\) is the complete graph
and \(w_G = \ones\), we have \(\rank(X^*) + \rank(S^*) = 1 + n_H\).

\Cref{thm:cosum} shows that, if \(F\) is a graph which is a cosum
(i.e., the complement of~\(F\) is not connected) \(F = G \cosum H\),
where \(G\) has at least one edge and \(H\) is connected, then there
is a nonnegative weight function \(w \colon E(F) \to \Reals_{++}\) such
that strict complementarity fails for~\eqref{eq:maxcut-sdp} with \(C =
\tfrac{1}{4}\Lcal_F(w)\); one may just fix \(w_H \in
\Reals_{++}^{E(H)}\) arbitrarily, e.g., \(w_H = \ones\), and set \(w_G
\coloneqq M \ones\) for large enough~\(M\) so that \(n_{G}\mu_{G} >
n_{H}\mu_{H}\).  A natural question following from this is:
\begin{problem}
  Characterize the set of graphs for which there exists a positive
  weight function on the edges such that strict complementarity fails
  for~\eqref{eq:maxcut-sdp} when \(4C\) is the corresponding weighted
  Laplacian matrix.
\end{problem}

\section{Generic Failure of Strict Complementarity on the Boundaries
  of Normal Cones}
\label{sec:generic-sc-failure}

In this section, we consider how often strict complementarity holds
for~\eqref{eq:maxcut-sdp} when \(C\) lies in the (relative) boundary
of \(\Normal{\Elliptope{n}}{\Xb}\) for some vertex \(\Xb\)
of~\(\Elliptope{n}\).  Note that this boundary is described as a
Minkowski sum in~\eqref{eq:bd-normal-vx}.

We start by considering the case \(n = 3\),
where~\eqref{eq:bd-normal-vx} simplifies to
\begin{equation}
  \label{eq:bd-normal-vx-3}
  \bd(\Normal{\Elliptope{3}}{\oprodsym{\xb}})
  =
  \Image(\Diag) - \setst{\oprodsym{z}}{z \in \set{\xb}^{\perp}}
\end{equation}
for every \(\xb \in \set{\pm1}^3\).

\begin{proposition}
  \label{prop:sc-fail3}
  Let \(\xb \in \set{\pm1}^3\), and let \(C = \Diag(\yb) -
  \oprodsym{\zb}\) for some \(\yb \in \Reals^3\) and \(\zb \in
  \set{\xb}^{\perp}\), so that \(C \in
  \bd(\Normal{\Elliptope{3}}{\oprodsym{\xb}})\).  Then strict
  complementarity holds for~\eqref{eq:maxcut-sdp} if and only if
  \(\zb_i = 0\) for some \(i \in [3]\).
\end{proposition}

\begin{proof}
  Set \(\Sb \coloneqq \Diag(\yb) - C = \oprodsym{\zb}\) and \(\Xb
  \coloneqq \oprodsym{\xb}\).  Clearly, \(\dualpair{\yb}{\Sb}\) is
  feasible in the dual and \(\Tr(\Sb\Xb) = (\iprodt{\zb}{\xb})^2 =
  0\), so \(\pdpair{\Xb}{\dualpair{\yb}{\Sb}}\) is a pair of
  primal-dual optimal solutions.  By
  Corollary~\ref{cor:unique-maxcut}, \(\dualpair{\yb}{\Sb}\) is the
  unique optimal solution in the dual.

  Suppose that \(\zb_i \neq 0\) for every \(i \in [3]\).  We claim
  that \(\Xb\) is the unique optimal solution in the primal.  Indeed,
  let \(X \in \Elliptope{3}\) be optimal in the primal.  Then \(0 =
  \Tr(\Sb X) = \qform{X}{\zb}\) so \(X\zb = 0\).  Thus,
  \[
  0 =
  \begin{bmatrix}
    1 & X_{12} & X_{13} \\
    X_{12} & 1 & X_{23} \\
    X_{13} & X_{23} & 1
  \end{bmatrix}
  \begin{bmatrix}
    \zb_1 \\ \zb_2 \\ \zb_3
  \end{bmatrix}
  =
  \begin{bmatrix}
    \zb_1 + \zb_2 X_{12} + \zb_3 X_{13} \\
    \zb_1 X_{12} + \zb_2 + \zb_3 X_{23} \\
    \zb_1 X_{13} + \zb_2 X_{23} + \zb_3
  \end{bmatrix},
  \]
  so
  \[
  \begin{bmatrix}
    \zb_2 & \zb_3 & 0 \\
    \zb_1 & 0 & \zb_3 \\
    0 & \zb_1 & \zb_2
  \end{bmatrix}
  \begin{bmatrix}
    X_{12} \\ X_{13} \\ X_{23}
  \end{bmatrix}
  =
  -\zb.
  \]
  The determinant of the matrix defining the latter linear system is
  \(-2\zb_1\zb_2\zb_3 \neq 0\), so the unique solution is given by the
  off-diagonal entries of \(\Xb\).

  Suppose now that \(\zb_i = 0\) for some \(i \in [3]\).  If \(\zb =
  0\) then \(\pdpair{I}{\dualpair{\yb}{0}}\) satisfies strict
  complementarity, so assume \(\zb \neq 0\).  Set \(\xt \coloneqq
  \Diag(\ones-e_i)\xb\) and \(\Xt \coloneqq \oprodsym{\xt} +
  \oprodsym{e_i} \in \Elliptope{3}\).  Then \(\Tr(\Sb\Xt) =
  \qform{(\oprodsym{\xt} + \oprodsym{e_i})}{\zb} =
  (\iprodt{\zb}{\xt})^2 + \zb_i^2 = 0\) since \(\iprodt{\zb}{\xt} =
  \iprodt{\zb}{\xb} = 0\).  Hence,
  \(\pdpair{\Xt}{\dualpair{\yb}{\Sb}}\) is a strictly complementarity
  pair of primal-dual optimal solutions for~\eqref{eq:maxcut-sdp}.
\end{proof}

For \(n \geq 4\), characterization of strict complementarity
in~\eqref{eq:maxcut-sdp} is not as easily described.  However, we can
prove the following condition sufficient for the failure of strict
complementarity, which will turn out to be sufficient for our
purposes.

\begin{theorem}
  \label{thm:sc-failure-high-rank}
  Let \(n \geq 3\).  Let \(C = \Diag(\yb) - \Sb\) for some \(\yb \in
  \Reals^n\) and \(\Sb \in \Psd{n}\), so that \(C \in
  \bd(\Normal{\Elliptope{n}}{\oprodsym{\ones}})\).  Suppose that
  \(\Null(\Sb) = \linspan\set{\ones,h}\) for some \(h \in
  \set{\ones}^{\perp}\) and that \(h\) has at least three distinct
  coordinates.  Then strict complementarity fails
  for~\eqref{eq:maxcut-sdp}.
\end{theorem}

\begin{proof}
  Set \(y^* \coloneqq \yb\) and \(S^* \coloneqq \Diag(y^*) - C =
  \Sb\).  Set \(\Xb \coloneqq \oprodsym{\ones}\).  Clearly,
  \(\dualpair{y^*}{S^*}\) is feasible in the dual and \(\Tr(S^*X^*) =
  0\), so \(\pdpair{\Xb}{\dualpair{y^*}{S^*}}\) is a pair of
  primal-dual optimal solutions.  By
  Corollary~\ref{cor:unique-maxcut}, \(\dualpair{y^*}{S^*}\) is the
  unique optimal solution in the dual.  We shall prove that \(\Xb\) is
  the unique optimal solution in the primal.

  Let \(X \in \Elliptope{n}\) be an optimal solution in the primal.
  By complementary slackness, \(\Tr(XS^*) = 0\), so \(S^*X = 0\) and
  \(\Image(X) \subseteq \Null(S^*) = \linspan\set{\ones,h}\).  Hence,
  \(X = \alpha_1 \oprodsym{\ones} + \alpha_2 \oprodsym{h} + \alpha_3
  (\oprod{h}{\ones} + \oprod{\ones}{h})\) for some \(\alpha \in
  \Reals^3\).  Since \(\diag(X) = \ones\), we find that \(\alpha_1 +
  \alpha_2 h_i^2 + 2\alpha_3 h_i = 1\) for every \(i \in [n]\).  Let
  \(i,j,k \in [n]\) such that \(|\set{h_i,h_j,h_k}| = 3\).  Then
  \[
  \begin{bmatrix}
    1 & 2 h_i & h_i^2\thinspace \\[2pt]
    1 & 2 h_j & h_j^2\thinspace \\[2pt]
    1 & 2 h_k & h_k^2\thinspace \\
  \end{bmatrix}
  \begin{bmatrix}
    \alpha_1 \\
    \alpha_3 \\
    \alpha_2
  \end{bmatrix}
  = \ones. \] The determinant of the matrix defining this linear
  system is a Vandermonde determinant, and it is equal to \(2^3 (h_j -
  h_i) (h_k - h_i) (h_k - h_j) \neq 0\) by assumption.  Hence,
  \(\alpha = e_1\) is its unique solution.  Thus, \(X =
  \oprodsym{\ones}\).
\end{proof}

\Cref{thm:sc-failure-high-rank} seems to indicate that strictly
complementarity fails ``almost everywhere'' on the boundary of
\(\Normal{\Elliptope{n}}{\oprodsym{\ones}}\), since the high rank
matrices make up the bulk of the boundary (consider that the set of
nonsingular matrices is open and dense) and for ``most'' of them the
extra vector \(h\) in the nullspace has at least three distinct
coordinates.  Unfortunately, we are dealing with somewhat complicated
sets (e.g., the high rank matrices in the boundary of a normal cone).
In order to make our previous statements precise, we shall make use of
the theory of Hausdorff measures, which we introduce next.

\subsection{Preliminaries on Hausdorff Measures}

We refer the reader to~\cite{Rogers98a}, though we use different
notation and more standard terminology.  See
also~\cite{DrusvyatskiyL11a,PatakiT01a} for a somewhat similar presentation.  We
focus our presentation on finite-dimensional normed spaces (over the
reals) but most of it could be developed for arbitrary metric spaces.
Our main normed spaces are (subspaces of) \(\Reals^n\) and
\(\Sym{n}\).  Since these are Euclidean spaces, they are equipped with
a norm induced by their inner-products, and that is the norm that we
will consider unless explicitly stated otherwise.  We shall only use
other norms in~\Cref{sec:rank-one2}.

Let \(\Vscr\) be a finite-dimensional normed space.  Let \(d \in
\Reals_+\) and \(\eps \in \Reals_{++}\).  For each \(\Xscr \subseteq
\Vscr\), define
\[
H_d^{\eps}(\Xscr)
\coloneqq
\inf\setst*{
  \sum_{i=0}^{\infty}
  \sqbrac[\big]{\diam(\Uscr_i)}^d
}{
  \set{\Uscr_i}_{i \in \Naturals} \subseteq \Powerset{\Vscr},\,
  \Xscr \subseteq \bigcup_{i=0}^{\infty} \Uscr_i,\,
  \diam(\Uscr_i) < \eps\,\forall i \in \Naturals
},
\]
where the \emph{diameter} of \(\Uscr \subseteq \Vscr\) is
\(\diam(\Uscr) \coloneqq \sup_{x,y \in \Uscr} \norm{x-y}\).  The
function \(H_d \colon \Powerset{\Vscr} \to [0,+\infty]\) defined by
\begin{equation}
  \label{eq:15}
  H_d(\Xscr)
  \coloneqq
  \sup_{\mathclap{\eps>0}} H_d^{\eps}(\Xscr)
  =
  \lim_{\eps \downarrow 0} H_d^{\eps}(\Xscr)
  \qquad
  \forall \Xscr \subseteq \Vscr
\end{equation}
is an outer measure on~\(\Vscr\).
Hence, the restriction of~\(H_d\) to the \(H_d\)-measurable subsets
of~\(\Vscr\) is a complete measure on~\(\Vscr\), called the
\emph{\(d\)-dimensional Hausdorff measure} on~\(\Vscr\).  The
\(0\)-dimensional Hausdorff measure~\(H_0\) is the cardinality of a
set, \(H_1\) is its length, \(H_2\) is its area, and so on.

Let \(d\) be a positive integer and set \(\Vscr \coloneqq \Reals^d\).
Let \(\lambda_d \colon \Powerset{\Reals^d} \to [0,+\infty]\) denote the
\(d\)-dimensional Lebesgue outer measure on~\(\Reals^d\).  It can be
proved~\cite[Theorem~30]{Rogers98a} that
\begin{equation}
  \label{eq:Lebesgue-Hausdorff}
  \frac{\lambda_d(\Xscr)}{\lambda_d(\Ball)}
  =
  \frac{H_d(\Xscr)}{2^d}
  \qquad
  \forall \Xscr \subseteq \Reals^d.
\end{equation}
In~particular, the \(H_d\)-measurable subsets of~\(\Reals^d\) are the
same as the \(\lambda_d\)\nobreakdash-measurable sets.

Let \(a,b \in \Reals_+\) with \(a < b\) and let \(\Xscr \subseteq \Vscr\).
It is not hard to prove from the definition that
\begin{alignat}{2}
  \label{eq:16}
  H_a(\Xscr) < \infty & \implies H_{b}(\Xscr) = 0, \\
  H_b(\Xscr) > 0      & \implies H_{a}(\Xscr) = \infty.
\end{alignat}
Hence,
\begin{equation}
  \label{eq:17}
  \sup\setst{d \in \Reals_+}{H_d(\Xscr) = \infty}
  =
  \inf\setst{d \in \Reals_+}{H_d(\Xscr) = 0},
\end{equation}
and the common value in~\eqref{eq:17} is the \emph{Hausdorff dimension}
of~\(\Xscr\), denoted by \(\dim_H(\Xscr)\).  In~particular,
\begin{equation}
  \label{eq:18}
  \text{
    if~\(d \in \Reals_+\) and \(\Xscr \subseteq \Vscr\)
    satisfy \(H_d(\Xscr) \in
    (0,\infty)\), then \(\dim_H(\Xscr) = d\).%
  }
\end{equation}

We may now define \emph{genericity} precisely.  Let \(\Xscr\) be a subset of
a finite-dimensional normed space~\(\Vscr\).  Let \(P\) be a property that may hold or
fail for points in~\(\Xscr\), i.e., \(P(x)\) is either true or false
for each \(x \in \Xscr\).  We say that \(P\) \emph{holds generically
  on~\(\Xscr\)} if \(H_d(\setst{x \in \Xscr}{\text{\(P(x)\) is
    false}}) = 0\) for \(d \coloneqq \dim_H(\Xscr)\).  We say that
\(P\) \emph{fails generically on~\(\Xscr\)} if the negation of~\(P\)
holds generically on~\(\Xscr\).  In~\Cref{sec:generic-failure}, we
will use \Cref{thm:sc-failure-high-rank} to prove that strict
complementarity fails generically at the boundary of the normal cone
of any vertex of~\(\Elliptope{n}\), for \(n \geq 3\), modulo some
qualification on the ambient space.  In the remainder of this section
and in the next one, we will describe a few more measure-theoretic
tools that we shall use towards this goal.

Let \(\Vscr\) and~\(\Uscr\) be finite-dimensional normed spaces.  Let
\(\Xscr \subseteq \Vscr\).  Recall that a function \(\varphi \colon \Xscr \to
\Uscr\) is \emph{Lipschitz continuous} with Lipschitz constant \(L >
0\) if
\begin{equation}
  \label{eq:19}
  \norm{\varphi(x)-\varphi(y)} \leq L\norm{x-y}
  \qquad
  \forall x,y \in \Xscr.
\end{equation}
The following is well known and easy to prove:
\begin{theorem}
  \label{thm:Lipschitz}
  Let \(\Vscr\) and~\(\Uscr\) be finite-dimensional normed spaces.  Let \(\Xscr \subseteq
  \Vscr\) and \(d \in \Reals_+\).  Let \(\varphi \colon \Xscr \to \Uscr\) be
  Lipschitz continuous with Lipschitz constant~\(L\).  Then
  \begin{equation}
    \label{eq:20}
    H_d\paren[\big]{\varphi(\Xscr)} \leq L^d H_d(\Xscr).
  \end{equation}
\end{theorem}
\Cref{thm:Lipschitz} is especially useful to determine some Hausdorff
dimensions via bi-Lipschitz maps.  We recall the definition here.  Let
\(\Vscr\) and~\(\Uscr\) be finite-dimensional normed spaces.  Let \(\Xscr \subseteq \Vscr\),
and let \(\varphi \colon \Xscr \to \Uscr\) be a one-to-one function with
range \(\Yscr \coloneqq \varphi(\Xscr)\).  We say that \(\varphi\) is
\emph{bi-Lipschitz continuous} with Lipschitz constants \(L_1 > 0\)
and \(L_2 > 0\) if \(\varphi\) is Lipschitz continuous with Lipschitz
constant~\(L_1\) and \(\varphi^{-1} \colon \Yscr \to \Vscr\) is Lipschitz
continuous with Lipschitz constant~\(L_2\).
\begin{corollary}
  \label{cor:bi-Lipschitz}
  Let \(\Vscr\) and~\(\Uscr\) be finite-dimensional normed spaces.  Let \(\Xscr \subseteq
  \Vscr\) and \(d \in \Reals_+\).  Let \(\varphi \colon \Xscr \to \Uscr\) be
  bi-Lipschitz continuous with Lipschitz constants~\(L_1\)
  and~\(L_2\).  Then
  \begin{equation}
    \label{eq:21}
    L_2^{-d} H_d(\Xscr) \leq H_d(\varphi(\Xscr)) \leq L_1^d H_d(\Xscr).
  \end{equation}
  In~particular, if \(H_d(\Xscr) \in (0,\infty)\), then
  \(\dim_H(\varphi(\Xscr)) = d\).
\end{corollary}
This corollary may be used, for~instance, to regard any
\(d\)-dimensional Euclidean space~\(\Vscr\) as~\(\Reals^d\) by
considering the coordinate map \(\varphi \colon \Vscr \to \Reals^d\)
with respect to a fixed orthonormal basis of~\(\Vscr\).  Another
frequent use of \Cref{cor:bi-Lipschitz} goes as follows.  Equip the
space \(\Sym{n}\) with the trace inner-product.  If \(Q \in \Reals^{n
  \times n}\) is an orthogonal matrix, the map \(X \in \Sym{n} \mapsto
QXQ^{\transp}\) preserves inner-products, and hence norms and
distances; hence, the map is Lipschitz continuous with Lipschitz
constant~1.  Its inverse is \(X \in \Sym{n} \mapsto Q^{\transp}XQ\)
and so the map \(X \in \Sym{n} \mapsto QXQ^{\transp}\) is bi-Lipschitz
continuous with Lipschitz constants~\(1\) and~\(1\).

The next result is useful for determining the Hausdorff dimension of
some simple unbounded sets in the \(\sigma\)-finite case,
when~\eqref{eq:18} is not directly applicable:
\begin{proposition}
  \label{prop:dim-countable}
  Let \(\Xscr\) be a subset of a finite-dimensional normed~\(\Vscr\).  For each \(i
  \in \Naturals\), let \(\Yscr_i\) be a subset of a finite-dimensional normed
  space~\(\Uscr_i\), and let \(\varphi_i \colon \Yscr_i \to \Vscr\) be a
  Lipschitz continuous function with Lipschitz constant~\(L_i\).  If
  \(\Xscr \subseteq \bigcup_{i \in \Naturals} \varphi_i(\Yscr_i)\), then
  \(\dim_H(\Xscr) \leq \sup_{i \in \Naturals} \dim_H(\Yscr_i)\).
\end{proposition}
\begin{proof}
  Set \(d \coloneqq \sup_{i \in \Naturals} \dim_H(\Yscr_i)\).  Let \(\dbar
  > d\).  Then~\eqref{eq:17} yields \(H_{\dbar}(\Yscr_i) = 0\) for each \(i
  \in \Naturals\), so by \Cref{thm:Lipschitz} we have \(H_{\dbar}(\Xscr)
  \leq \sum_{i \in \Naturals} L_i^{\dbar} H_{\dbar}(\Yscr_i) = 0\).
\end{proof}
For instance, \(\Reals^d = \bigcup_{M \in \Naturals} M\Ball\) and the
ball \(M\Ball \subseteq \Reals^d\) with nonzero~\(M\) has Hausdorff
dimension~\(d\) by~\eqref{eq:18} and~\eqref{eq:Lebesgue-Hausdorff}, so
\Cref{prop:dim-countable} shows that \(\dim_H(\Reals^d) \leq d\).
Since \(\Reals^d \supseteq \Ball\) shows that \(H_d(\Reals^d) \geq
H_d(\Ball) > 0\) by~\eqref{eq:Lebesgue-Hausdorff}, we conclude by~\eqref{eq:18} that
\(\dim_H(\Reals^d) = d\).  Together with \Cref{cor:bi-Lipschitz}, this
shows that Hausdorff dimension and the usual (linear) dimension
coincide on linear subspaces, and hence also for convex sets by
translation invariance.

\subsection{Hausdorff Measures and the Boundary Structure of Convex Sets}
\label{sec:bdstruct}

In this section we collect some results relating Hausdorff measures
and the boundary structure of convex sets, including a quick review of
basic facts about faces.

The following result is well known:
\begin{theorem}
  \label{thm:dimH-bd-compact}
  Let \(\Ebb\) be an Euclidean space.  If \(\Cscr \subseteq \Ebb\) is a
  compact convex set with dimension \(d \geq 1\), then
  \(\dim_H(\rbd(\Cscr)) = d-1\).
\end{theorem}
\begin{proof}
  We may assume that \(\dim(\Ebb) = d\) so that \(\Cscr\) has nonempty
  interior.  By choosing an orthonormal basis for~\(\Ebb\), we may
  assume that \(\Ebb = \Reals^d\).  We may also assume that \(0 \in
  \interior(\Cscr)\) by translation invariance of Hausdorff measure.
  Set \(X \coloneqq
  \bd(\Ball_{\infty})\), and note that \(H_{d-1}(X) \in (0,+\infty)\)
  by~\eqref{eq:Lebesgue-Hausdorff} and \Cref{cor:bi-Lipschitz}.  Let
  \(\eps,M \in \Reals_{++}\) such that \(2 \eps \Ball_{\infty}
  \subseteq \Cscr \subseteq \tfrac{1}{2}M\kern .5pt \Ball_{\infty}\).
  Let \(p_\Cscr \colon \Reals^d \to \Cscr\) be the metric projection
  onto~\(\Cscr\), i.e., \(\set{p_\Cscr(x)} = \argmin_{y \in
    \Cscr}\norm{y-x}\) for each \(x \in \Reals^d\).  Then \(p_\Cscr\)
  is Lipschitz continuous (with Lipschitz constant~1).
  \Cref{thm:Lipschitz} applied
  to~\(p_\Cscr\mathord{\restriction}_{MX}\) and positive homogeneity
  of~\(H_{d-1}\) (of degree~\(d-1\)) yield \(H_{d-1}(\bd(\Cscr)) <
  \infty\).  Similarly, applying \Cref{thm:Lipschitz} to the
  restriction to~\(\bd(\Cscr)\) of metric projection onto~\(\eps
  \Ball_{\infty}\) yields \(H_{d-1}(\bd(\Cscr)) > 0\).  The theorem
  now follows from~\eqref{eq:18}.
\end{proof}

Since we are dealing with convex cones, the previous result will be
more useful to us when stated in a lifted form about pointed closed
convex cones:
\begin{corollary}
  \label{cor:dimH-bd-cone}
  Let \(\Ebb\) be an Euclidean space.  If \(\Kscr \subseteq \Ebb\) is a
  pointed closed convex cone with dimension \(d \geq 1\), then
  \(\dim_H(\rbd(\Kscr)) = d-1\).
\end{corollary}
\begin{proof}
  We may assume that \(\Ebb = \Reals^d\).  Since \(\Kscr\) is pointed,
  after applying some rotation, which preserves Hausdorff measures by
  \Cref{cor:bi-Lipschitz}, we may assume that \(\Kscr = \Reals_+(1
  \oplus \Cscr)\) for some compact convex set \(\Cscr \subseteq
  \Reals^{\dbar}\) where \(\dbar \coloneqq d-1\).  For each \(N \in
  \Naturals\), define the compact convex set \(\Kscr_N \coloneqq \Kscr
  \cap [N,N+1] \oplus \Reals^{\dbar}\).
  Since
  \begin{equation}
    \label{eq:22}
    \rbd(\Kscr)
    \subseteq
    \bigcup_{N=0}^{\infty}
    \rbd(\Kscr_N),
  \end{equation}
  the result follows from \Cref{prop:dim-countable} and
  \Cref{thm:dimH-bd-compact}.
\end{proof}

The next result refers to faces of a convex set, so before we state it
we shall briefly recall the basic theory;
see~\cite[Sec.~18]{Rockafellar97a}.  Let \(\Ebb\) be an Euclidean
space.  Let \(\Cscr \subseteq \Ebb\) be a convex
set.  A convex subset~\(\Fscr\) of~\(\Cscr\) is a \emph{face} of~\(\Cscr\) if, for
each \(x,y \in \Cscr\) such that the open line segment \((x,y) \coloneqq
\setst{(1-\lambda)x+\lambda y}{\lambda \in (0,1)}\) between~\(x\)
and~\(y\) meets~\(\Fscr\), we have \(x,y \in \Fscr\).  We use the notation \(\Fscr
\faceeq \Cscr\) to denote that~\(\Fscr\) is a face of~\(\Cscr\), and \(\Fscr \faceneq
\Cscr\) to denote that~\(\Fscr\) is a \emph{proper} face of~\(\Cscr\), i.e., \(\Fscr
\faceeq \Cscr\) and \(\Fscr \neq \Cscr\).  Denote \(\Faces(\Cscr) \coloneqq
\setst{\Fscr}{\Fscr \faceeq \Cscr}\).

Faces of closed convex sets are closed, and faces of convex cones are
convex cones.  An arbitrary intersection of faces of~\(\Cscr\) is a face
of~\(\Cscr\) and, since the faces of a convex set are partially ordered by
inclusion and \(\Cscr \faceeq \Cscr\), every point~\(x\) of~\(\Cscr\) lies in a
unique minimal face~\(\Fscr\) of~\(\Cscr\); this face~\(\Fscr\) is characterized
by the property \(x \in \ri(\Fscr)\).  Also, it can be proved that
\(\setst{\ri(\Fscr)}{\emptyset \neq \Fscr \faceeq \Cscr}\) is a partition
of~\(\Cscr\).  If \(\Cscr\) is a compact convex set, it is not hard to prove
that the faces of the homogenization of~\(\Cscr\) are described by:
\begin{equation}
  \label{eq:25}
  \Faces\paren[\big]{
    \Reals_+(1 \oplus \Cscr)
  }
  =
  \set[\big]{\emptyset,\set{0}}
  \cup
  \setst[\big]{
    \Reals_+(1 \oplus \Fscr)
  }{
    \emptyset \neq \Fscr \faceeq \Cscr
  }.
\end{equation}

\begin{theorem}[Larman~\cite{Larman71a}]
  \label{thm:Larman}
  Let \(\Ebb\) be an Euclidean space.  If \(\Cscr \subseteq \Ebb\) is a
  compact convex set with dimension \(d \geq 1\), then
  \begin{equation*}
    H_{d-1}\paren[\Big]{\,
      \bigcup_{\Fscr \faceneq \Cscr}{\rbd(\Fscr)}
    } = 0.
  \end{equation*}
\end{theorem}

As before, we shall need a conic version of Larman's Theorem.
We apply tools similar to the ones used to lift
\Cref{thm:dimH-bd-compact} to \Cref{cor:dimH-bd-cone}:
\begin{theorem}
  \label{thm:conic-Larman}
  Let \(\Ebb\) be an Euclidean space.  If \(\Kscr \subseteq \Ebb\) is a
  pointed closed convex cone with dimension \(d \geq 1\), then
  \begin{equation*}
    H_{d-1}\paren[\Big]{\,
      \bigcup_{\Fscr \faceneq \Kscr}{\rbd(\Fscr)}
    } = 0.
  \end{equation*}
\end{theorem}
\begin{proof}
  The case \(d = 1\) is easy to verify; assume that \(d \geq 2\).  We
  may assume that \(\Ebb = \Reals \oplus \Reals^{\dbar}\) for \(\dbar
  \coloneqq d-1\) and, as in the beginning of the proof of
  \Cref{cor:dimH-bd-cone}, we may assume that \(\Kscr =
  \Reals_+(1\oplus \Cscr)\) for some compact convex set \(\Cscr
  \subseteq \Reals^{\dbar}\) with nonempty interior.  For each \(N \in
  \Naturals\), define the compact convex set \(\Kscr_N \coloneqq \Kscr
  \cap [N,N+1] \oplus \Reals^{\dbar}\).  By elementary convex analysis,
  \begin{equation}
    \label{eq:63}
    \bigcup_{\Fscr \faceneq \Kscr} \rbd(\Fscr)
    \subseteq
    \bigcup_{N=0}^{\infty}
    \bigcup_{\Fscr_N \faceneq \Kscr_N}
    \rbd(\Fscr_N).
  \end{equation}
  Hence,
  \begin{equation*}
    H_{d-1}\paren*{
      \bigcup_{\Fscr \faceneq \Kscr} \rbd(\Fscr)
    }
    \leq
    \sum_{N=0}^{\infty}
    H_{d-1}\paren*{
      \bigcup_{\Fscr_N \faceneq \Kscr_N}
      \rbd(\Fscr_N)
    }
    = 0,
  \end{equation*}
  where we used the fact that each summand is zero by
  \Cref{thm:Larman}.
\end{proof}

\subsection{Generic Failure of Strict Complementarity}
\label{sec:generic-failure}

In this section, we prove one of our main results: strict
complementarity fails generically in the relative boundary of the
normal cone of the elliptope at any of its vertices.

We shall apply \Cref{thm:conic-Larman} to~\(\Psd{n}\).  Let us briefly
recall some well-known descriptions of the faces of the positive
semidefinite cone~\(\Psd{n}\).  Let \(\Lfrak_n\) denote the set of
linear subspaces of~\(\Reals^n\).  For each \(\Lscr \in \Lfrak_n\), denote
\begin{equation}
  \label{eq:26}
  \Fscr_\Lscr
 \coloneqq
 \setst{
   X \in \Psd{n}
 }{
   \Null(X) \supseteq \Lscr
 }
\end{equation}
and note that
\begin{equation}
  \label{eq:27}
  \ri(\Fscr_\Lscr)
  =
  \setst{
    X \in \Psd{n}
  }{
    \Null(X) = \Lscr
  }.
\end{equation}
Then
\begin{gather}
  \label{eq:28}
  \Faces(\Psd{n})
  =
  \set{\emptyset}
  \cup
  \setst{
    \Fscr_\Lscr
  }{
    \Lscr \in \Lfrak_n
  }.
\end{gather}
Note that, for \(\Lscr \in \Lfrak_n\) such that \(\Lscr \neq \Reals^n\), there
is an orthogonal matrix \(Q \in \Reals^{n \times n}\) such that
\begin{equation}
  \label{eq:29}
  \Fscr_\Lscr
  =
  \setst[\bigg]{
    Q
    \begin{bmatrix}
      U & 0 \\
      0 & 0 \\
    \end{bmatrix}
    Q^{\transp}
  }{
    U \in \Psd{r}
  },
\end{equation}
where \(r \coloneqq n - \dim(\Lscr)\).

\begin{lemma}
  \label{lem:low-rank-negligible}
  Let \(n \geq 2\) be an integer.  Then the property ``\(C \mapsto
  \rank(C) = n-1\)'' holds generically in \(\bd(\Psd{n})\).
\end{lemma}
\begin{proof}
  Set \(d \coloneqq \dim_H(\Psd{n})\).  Note that \(d-1 =
  \dim_H(\bd(\Psd{n}))\) by \Cref{cor:dimH-bd-cone}.  Let \(X \in
  \bd(\Psd{n})\) such that \(\rank(X) = n-1\) fails.  Then \(\rank(X)
  \leq n-2\).  For each nonzero \(h \in \Null(X)\), let \(\Lscr\) be the
  linear subspace of~\(\Reals^n\) spanned by~\(h\) and note that \(X
  \in \rbd(\Fscr_\Lscr)\), following the notation from~\eqref{eq:26}.  Hence,
  \begin{equation*}
    \setst{
      X \in \bd(\Psd{n})
    }{
      \rank(X) \neq n-1
    }
    =
    \setst{
      X \in \Psd{n}
    }{
      \rank(X) \leq n-2
    }
    \subseteq
    \bigcup_{\mathclap{\Fscr \faceneq \Psd{n}}} \rbd(\Fscr).
  \end{equation*}
  The \((d-1)\)-dimensional Hausdorff measure of the set on the RHS
  above is zero by \Cref{thm:conic-Larman}.
\end{proof}

We are ready to prove one of our main results:
\begin{theorem}
  \label{thm:generic-sc-failure}
  Let \(n \geq 3\), and let \(\Xb\) be a vertex of~\(\Elliptope{n}\).
  Then the property ``\(C \mapsto\) strict complementarity holds for
  \eqref{eq:maxcut-sdp}'' fails generically on \(\rbd(\Psd{n} \cap
  \set{\Xb}^{\perp})\).
\end{theorem}
\begin{proof}
  By \Cref{thm:vertices-elliptope} and the discussion of linear
  automorphisms of~\(\Elliptope{n}\) from \Cref{sec:vertices}, we may
  assume that \(\Xb = \oprodsym{\ones}\).  Set \[m \coloneqq n-1.\]
  Let \(Q \in \Reals^{n \times n}\) be an orthogonal matrix such that
  \(Q^{\transp} e_n = n^{-1/2}\ones\) and \(Q^{\transp} e_{m} =
  2^{-1/2}(e_1-e_2)\).  Using the map \(M \in \Sym{n} \mapsto
  QMQ^{\transp}\) and \Cref{cor:bi-Lipschitz}, we find that
  \(\rbd(\Psd{n} \cap \set{\Xbar}^{\perp})\) and \(\rbd(\Psd{n} \cap
  \set{\oprodsym{e_n}}^{\perp})\) have the same Hausdorff dimension.
  Since the cone \(\Psd{n} \cap \set{\oprodsym{e_n}}^{\perp}\) is an
  embedding of \(\Psd{m}\) into~\(\Psd{n}\), the Hausdorff dimension
  of \(\rbd(\Psd{n} \cap \set{\oprodsym{e_n}}^{\perp})\) is
  \(\dim_H(\Psd{m})-1\) by \Cref{cor:dimH-bd-cone}.  Hence,
  \begin{equation}
    \label{eq:30}
    d \coloneqq
    \dim_H\paren*{
      \rbd\paren*{
        \Psd{n} \cap \set{\Xbar}^{\perp}
      }
    }
    = \binom{n}{2}-1.
  \end{equation}

  Set
  {\(
    \Cscr \coloneqq \setst[\big]{
      C \in \rbd(\Psd{n} \cap \set{\Xbar}^{\perp})
    }{
      \text{strict complementarity holds in~\eqref{eq:maxcut-sdp}}
    }\)}.
  By \Cref{thm:sc-failure-high-rank},
  \begin{equation}
    \label{eq:31}
    \Cscr \subseteq
    \Dscr_0
    \cup
    \Dscr_{12} \cup \Dscr_{13} \cup \Dscr_{23}
  \end{equation}
  where
  \begin{gather*}
    \Dscr_0 \coloneqq
    \setst[\big]{
      C \in \Psd{n} \cap \set{\Xbar}^{\perp}
    }{
      \rank(C) \leq n-3
    },
    \\
    \Dscr_{ij} \coloneqq
    \setst[\big]{
      C \in \Psd{n}
    }{
      \exists h \in \set{\ones,e_i-e_j}^{\perp},\,
      h \neq 0,\,
      \Null(C) = \linspan\set{\ones,h}
    },
  \end{gather*}
  for each \(i,j \in [n]\).  Clearly all the sets \(\Dscr_{ij}\) have
  the same \(d\)-dimensional Hausdorff measures, so it suffices to
  prove that
  \begin{gather}
    \label{eq:32}
    H_d(\Dscr_0) = 0,
    \\
    \label{eq:33}
    H_d(\Dscr_{12}) = 0.
  \end{gather}

  By using the map \(M \in \Sym{n} \mapsto QMQ^{\transp}\) and
  \Cref{cor:bi-Lipschitz}, \(\Dscr_0\) and \(\setst{C \in
    \Psd{m}}{\rank(C) \leq m-2}\) have the same \(d\)-dimensional
  Hausdorff measure.  Hence, \eqref{eq:32} follows from
  \Cref{lem:low-rank-negligible} and \Cref{cor:dimH-bd-cone}.  Again
  using the map \(M \in \Sym{n} \mapsto QMQ^{\transp}\) and
  \Cref{cor:bi-Lipschitz}, we find that \(H_d(\Dscr_{12}) =
  H_d(\Dscr')\) where
  \begin{equation*}
    \Dscr' \coloneqq
    \setst{
      U \in \Psd{m}
    }{
      \rank(U) = m-1, e_m \in \Image(U)
    }.
  \end{equation*}
  Hence, to prove~\eqref{eq:33} and thus the theorem, it suffices to
  prove that
  \begin{equation}
    \label{eq:34}
    H_d(\Dscr') = 0.
  \end{equation}

  For each \(k \in [m-1]\) define the permutation matrix \(P_k
  \coloneqq \sum_{i \in [m] \setminus \set{k,m}} \oprodsym{e_i} +
  \oprod{e_k}{e_m} + \oprod{e_m}{e_k} \in \Sym{m}\).  Set \(P_m
  \coloneqq I\).  For each \(k \in [m]\) define the map \(\varphi_k
  \colon \Pd{m-1} \oplus \Reals^{m-1} \to \Sym{m}\) by setting
  \begin{equation*}
    \varphi_k(A \oplus c)
    \coloneqq
    P_k^{\transp}
    \begin{bmatrix}
      A             & Ac           \\
      c^{\transp} A & \qform{A}{c} \\
    \end{bmatrix}
    P_k.
  \end{equation*}
  It is easy to verify that
  \begin{gather}
    \label{eq:35}
    \setst{U \in \Psd{m}}{\rank(U) = m-1}
    =
    \bigcup_{k \in [m]} \varphi_k(\Pd{m-1} \oplus \Reals^{m-1}),
    \\
    \Null(\varphi_k(A \oplus c))
    =
    P_k\linspan\set{-c \oplus 1}
    \qquad
    \forall A \oplus c \in \Pd{m-1} \oplus \Reals^{m-1}.
  \end{gather}
  Let \(U \in \Psd{m}\) with \(\rank(U) = m-1\), and let \(k \in [m]\)
  and \(A \oplus c \in \Pd{m-1} \oplus \Reals^{m-1}\) such that \(U =
  \varphi_k(A \oplus c)\).  Then \(e_m \in \Image(U)\) is equivalent
  to \(e_m \perp P_k (-c \oplus 1)\), which is equivalent to \(k \in
  [m-1]\) and \(c \perp e_k\).  Hence,
  \begin{equation}
    \label{eq:36}
    \Dscr'
    =
    \bigcup_{\mathclap{k \in [m-1]}} \varphi_k(\Pd{m-1} \oplus \set{e_k}^{\perp}).
  \end{equation}
  Let \(k \in [m-1]\).  Since each entry of \(\varphi_k(A \oplus
  c)\) is (component-wise) polynomial function of the input, the map \(\varphi_k\) is
  Lipschitz continuous on any compact subset of the domain.  It
  follows from \Cref{prop:dim-countable} that
  \begin{equation}
    \label{eq:37}
    \dim_H(\varphi(\Pd{m-1} \oplus \set{e_k}^{\perp}))
    \leq
    \binom{m}{2} + m - 2
    =
    d - 1;
  \end{equation}
  note that the subspace \(\set{e_k}^{\perp}\) in the LHS is
  \((m-2)\)-dimensional, as this subspace is the set of vectors
  in~\(\Reals^{m-1}\) orthogonal to~\(e_k\).  Now~\eqref{eq:34}
  follows from~\eqref{eq:36} and~\eqref{eq:37}.
\end{proof}

\section{Failure of Strict Complementarity for Rank-One Objectives}
\label{sec:rank-one2}

In \Cref{sec:generic-sc-failure}, we zoomed into the boundary of the
normal cone of an arbitrary vertex of the elliptope and proved that
strict complementarity fails generically there.  \emph{Informally}, we
might say that with zero ``probability'' a ``uniformly chosen''
objective function in the boundary of such normal cone yields an SDP
that satisfies strict complementarity.  Now we zoom in even
further in that boundary, into the set of negative semidefinite
rank-one objectives, and consider again how often strict
complementarity holds.  We will state and prove a self-contained 
result in \Cref{thm:rank-one} below.  However, in order to motivate
the objects of the construction and the intermediate results, we
start with an informal discussion.

Assume throughout this discussion that \(n \geq 4\).  We will
normalize the ``sample space'' so that we can have a
probability space.  For the sake of discussion, let us focus our
attention on the vertex \(\oprodsym{\ones}\) of~\(\Elliptope{n}\) and
consider the sample space to be
\begin{equation}
  \Omega_M
  \coloneqq
  \setst{
    C \in \bd(\Normal{\Elliptope{n}}{\oprodsym{\ones}})
  }{
    C \preceq 0,\,
    \rank(C) = 1,\,
    \pnorm[\infty]{\matvec(C)} = 1
  }.
\end{equation}
Accordingly, equip \(\Sym{n}\) with the norm \(X \in \Sym{n} \mapsto
\pnorm[\infty]{\matvec(X)}\).  Set \(d \coloneqq \dim_H(\Omega_M)\).
In~order to obtain a probability space on~\(\Omega_M\), we will define
a probability measure
\begin{equation}
\label{eq:38}
  \Prob_M(\Ascr_M)
  \coloneqq
  \frac{H_d(\Ascr_M)}{H_d(\Omega_M)}
\end{equation}
over all \(H_d\)-measurable subsets~\(\Ascr_M\) of~\(\Omega_M\); we
shall prove that \(H_{n-2}(\Omega_M) \in (0,\infty)\), so
that~\eqref{eq:38} is properly defined and \(d=n-2\).  Our goal is to
prove that the probability of the event
\begin{equation}
\label{eq:39}
  \Gscr_M
  \coloneqq
  \setst{
    C \in \Omega_M
  }{
    \text{strict complementarity holds for~\eqref{eq:maxcut-sdp} with~\(C\)}
  }.
\end{equation}
lies in \((0,1)\).

In order to achieve this, we shall reduce the problem to the space of
vectors that generate the rank-one tensors in~\(\Omega_M\) and
\(\Gscr_M\), which lie in the matrix space.  In order to carry results
back and forth between these spaces, we rely on
\Cref{cor:bi-Lipschitz}.  For each \(s \in \set{\pm1}^n\), define
\begin{gather}
  \label{eq:40}
  \Reals_s^n \coloneqq \Diag(s)\Reals_+^n,
  \\
  \label{eq:41}
  \varphi_s \colon b \in \Reals_s^n \cap \bd(\Ball_{\infty}) \mapsto -\oprodsym{b}.
\end{gather}
Equip \(\Reals^n\) with the norm \(x \in \Reals^n \mapsto
\pnorm[\infty]{x}\).  We shall split our analysis to each of the
\(2^n\) bi-Lipschitz maps~\(\varphi_s\), one for each chamber/orthant
of~\(\Reals^n\), according to their sign vectors:
\begin{theorem}
  \label{thm:bi-Lipschitz-tensor}
  Let \(s \in \set{\pm1}^n\).  Then the map \(\varphi_s\) defined
  in~\eqref{eq:41} is bi-Lipschitz continuous with Lipschitz
  constants~2 and~1, where we equip the domain with the
  \(\infty\)-norm, and we equip the range with the norm
  \(\pnorm[\infty]{\matvec(\cdot)}\).
\end{theorem}
\begin{proof}
  To see that \(\varphi_s\) is Lipschitz continuous with Lipschitz
  constant~2, let \(x,y \in \Reals_s^n \cap \bd(\Ball_{\infty})\) and
  note that
  \begin{equation*}
    \pnorm[\infty]{2\matvec(\oprodsym{x}-\oprodsym{y})}
    =
    \pnorm[\infty]{
      \matvec\sqbrac{\oprod{(x-y)}{(x+y)} + \oprod{(x+y)}{(x-y)}
      }
    }
    \leq
    2 \pnorm[\infty]{x+y}\pnorm[\infty]{x-y}
    \leq
    4 \pnorm[\infty]{x-y}.
  \end{equation*}
  The proof that \(\varphi_s^{-1}\) is Lipschitz continuous with
  Lipschitz constant~1 is also simple but it involves case analysis.
  Set \(A \coloneqq \oprodsym{x}-\oprodsym{y}\).  Let \(k \in [n]\)
  such that \(\abs{x_k} = 1\), so \(x_k = s_k\).  Similarly, let
  \(\ell \in [n]\) such that \(\abs{y_{\ell}} = 1\), so \(y_{\ell} =
  s_{\ell}\).  Let \(j \in [n]\).  We shall make use of the following
  facts:
  \begin{gather*}
    \alpha_k
    \coloneqq
    \frac{y_k}{s_k} \in [0,1],
    \qquad
    \beta_{\ell}
    \coloneqq
    \frac{x_{\ell}}{s_{\ell}} \in [0,1],
    \qquad
    \abs{A_{kj}}
    =
    \abs*{x_j - \alpha_k y_j},
    \qquad
    \abs{A_{\ell j}}
    =
    \abs*{\beta_{\ell} x_j - y_j}.
  \end{gather*}
  We consider 4 cases, according to which of \(x_j\) or \(y_j\) is
  largest, and according to their signs; note that both \(x_j\) and
  \(y_j\) have the same sign.

  We have
  \begin{align*}
    x_j \geq y_j \geq 0
    & \implies
    0 \leq \abs{x_j-y_j} = x_j - y_j \leq x_j - \alpha_k y_j = \abs{A_{kj}};
    \\
    y_j \geq x_j \geq 0
    & \implies
    0 \leq \abs{x_j-y_j} = y_j - x_j \leq y_j - \beta_{\ell} x_j = \abs{A_{\ell j}};
    \\
    0 \geq x_j \geq y_j
    & \implies
    0 \leq \abs{x_j-y_j} = x_j - y_j \leq \beta_{\ell} x_j - y_j = \abs{A_{\ell j}};
    \\
    0 \geq y_j \geq x_j
    & \implies
    0 \leq \abs{x_j-y_j} = y_j - x_j \leq \alpha_k y_j - x_j = \abs{A_{kj}}.
  \end{align*}
  Hence, \(\pnorm[\infty]{x-y} \leq
  \pnorm[\infty]{\matvec(\oprodsym{x}-\oprodsym{y})}\).
\end{proof}

Note that restricting the domain of~\(\varphi_s\) in
\Cref{thm:bi-Lipschitz-tensor} to chambers of~\(\Reals^n\) is
necessary.  Indeed, consider \(x \coloneqq (1,-1,\eps)^{\transp}\) and
\(y \coloneqq (-1,1,0)^{\transp}\), for an arbitrary \(\eps \in
(0,1)\).  Then \(\pnorm[\infty]{x-y} = 2\) but
\(\pnorm[\infty]{\matvec(\oprodsym{x}-\oprodsym{y})} = \eps\).

Next we relate the description for~\(\Omega_M\) to the vectors that
appear in the rank-one tensors:
\begin{proposition}
  \label{prop:omega}
  For \(n \geq 3\), we have
  \begin{equation}
    \label{eq:42}
    \Omega_M
    =
    \setst*{
      -\oprodsym{b}
    }{
      b \in \Reals^n
      \text{ and }
      \begin{array}[!h]{l}
        \text{%
          either
          \(b = e_i - \alpha e_j\) for some distinct \(i,j \in [n]\)
          and \(\alpha \in [0,1]\),
        }
        \\
        \text{or }
        (b \perp \ones \text{ and }
        \card{\supp(b)} \geq 3 \text{ and }
        \pnorm[\infty]{b} = 1)
      \end{array}
    }
  \end{equation}
\end{proposition}
\begin{proof}
  We first prove the inclusion `\(\supseteq\)'.  If \(b \perp \ones\)
  and \(\pnorm[\infty]{b} = 1\), it follows from~\eqref{eq:normal}
  that \(-\oprodsym{b} \in \Omega_M\).  Suppose that \(b = e_i -
  \alpha e_j\) for distinct \(i,j \in [n]\) and \(\alpha \in [0,1]\).
  Set \(\beta \coloneqq 1 - \alpha \in [0,1]\) and \(y \coloneqq
  -\beta b\).  It is easy to verify that \(S \coloneqq \Diag(y) +
  \oprodsym{b} \succeq 0\) and \(S\ones = 0\); now \(-\oprodsym{b} =
  \Diag(y) - S \in \Omega_M\) follows from~\eqref{eq:normal}.  In both
  cases, we rely on \(n \geq 3\) to ensure that \(-\oprodsym{b}\) lies
  in the boundary.

  Now we prove the inclusion `\(\subseteq\)'.  Let \(b \in \Reals^n\)
  such that \(-\oprodsym{b} \eqqcolon C \in \Omega_M\).  Clearly
  \(\pnorm[\infty]{b} = 1\).  We may assume that \(\beta \coloneqq
  \iprodt{\ones}{b} \geq 0\) and that \(b_1 > 0\).
  Use~\eqref{eq:normal} to write \(C = \Diag(y) - S\) for some \(y \in
  \Reals^n\) and \(S \in \Psd{n}\) such that \(S\ones = 0\).  Then
  \(-\beta b = -\oprodsym{b}\ones = C\ones = y - S\ones = y\), so
  \begin{equation}
    \label{eq:43}
    0 \preceq S = \Diag(y) + \oprodsym{b} = -\beta \Diag(b) + \oprodsym{b}.
  \end{equation}
  We claim that
  \begin{equation}
    \label{eq:44}
    b_i < 0
    \qquad
    \forall i \in \supp(b) \setminus \set{1}.
  \end{equation}
  Indeed, by restricting~\eqref{eq:43} to a principal submatrix we get
  \begin{equation}
    \begin{bmatrix}
      b_1^2   & b_1 b_i \\
      b_1 b_i & b_i^2   \\
    \end{bmatrix}
    \succeq
    \beta
    \begin{bmatrix}
      b_1 & 0   \\
      0   & b_i \\
    \end{bmatrix}.
  \end{equation}
  If \(b_i > 0\), then the RHS is positive definite, whereas the LHS
  is singular.  This proves~\eqref{eq:44}.

  Suppose first that \(\card{\supp(b)} \leq 2\).  Then \(b = e_1 -
  \alpha e_j\) for some \(j \in \supp(b) \setminus \set{1}\) and
  \(\alpha \in [-1,1]\).  By~\eqref{eq:44}, we have \(\alpha \in
  [0,1]\), and so \(-\oprodsym{b}\) lies in the RHS of~\eqref{eq:42}.

  Suppose next that \(\card{\supp(b)} \geq 3\).  We must prove that
  \begin{equation}
    \label{eq:45}
    b \perp \ones.
  \end{equation}
  Suppose for the sake of contradiction that \(\beta > 0\).  Next let
  \(i,j \in \supp(b) \setminus \set{1}\) be distinct.  Again
  by~\eqref{eq:43} we get that the determinant of
  \begin{equation}
    \begin{bmatrix}
      b_1(b_1-\beta) & b_1 b_i        \\
      b_1 b_i        & b_i(b_i-\beta) \\
    \end{bmatrix}
    \succeq 0,
  \end{equation}
  is nonnegative.  This yields \(b_1 + b_i \leq \beta\) using \(\beta
  > 0\).  But
  \eqref{eq:44} implies that \(\beta \leq b_1 + b_i + b_j < b_1 +
  b_i\), contradiction.  This concludes the proof of~\eqref{eq:45},
  and hence \(-\oprodsym{b}\) lies in the RHS of~\eqref{eq:42}.
\end{proof}

Finally, we need to relate \(\Gscr_M\) with the vectors that appear in
the rank-one tensors.  A vector \(b \in \Reals^n\) is \emph{strictly
  balanced} if \(\abs{b_i} < \sum_{j \in [n]\setminus\set{i}}
\abs{b_j}\) for every \(i \in [n]\).  It is easy to verify that,
\begin{equation}
  \label{eq:46}
  \text{%
    if \(b \in \Reals^n\) and \(i \in [n]\) is such that \(\abs{b_i} =
    \pnorm[\infty]{b}\), then \(b\) is strictly balanced \(\iff
    \abs{b_i} < \textstyle\sum_{j \in [n] \setminus \set{i}} \abs{b_j}\).
  }
\end{equation}
We shall rely on yet another result by Laurent and Poljak:
\begin{theorem}[{\cite[Theorem~2.6]{LaurentP96a}}]
  \label{thm:strictly-balanced}
  Let \(b \in \Reals^n\) such that \(b \perp \ones\) and \(\supp(b) =
  [n]\).  Then there exists \(X \in \Elliptope{n}\) such that
  \(\Null(X) = \linspan\set{b}\) if and only if \(b\) is strictly
  balanced.
\end{theorem}

\begin{proposition}
  \label{prop:sc-iff-sb}
  Let \(b \in \Reals^n\) such that \(b \perp \ones\) and \(\supp(b) =
  [n]\).  Then strict complementarity holds for~\eqref{eq:maxcut-sdp}
  with \(C = -\oprodsym{b}\) if and only if \(b\) is strictly
  balanced.
\end{proposition}
\begin{proof}
  Note that \(\oprodsym{\ones}\) is an optimal solution
  for~\eqref{eq:maxcut-sdp} if \(C = -\oprodsym{b}\).  By Proposition~\ref{prop:sc-ri}, we must
  show that existence of \(X \in \Elliptope{n}\) such that
  \(-\oprodsym{b} \in \ri(\Normal{\Elliptope{n}}{X})\) is equivalent
  to strict balancedness of~\(b\).  We will show that, for each \(X
  \in \Elliptope{n}\),
  \begin{equation}
    \label{eq:reduction-balancedness}
    -\oprodsym{b} \in \ri(\Normal{\Elliptope{n}}{X})
    \iff
    \oprodsym{b} \in \setst{Z \in \Psd{n}}{\Image(Z)=\Null(X)}.
  \end{equation}
  Since existence of \(X \in \Elliptope{n}\) such that the RHS
  of~\eqref{eq:reduction-balancedness} holds is equivalent to \(b\)
  being strictly balanced by Theorem~\ref{thm:strictly-balanced}, the
  result will follow.

  The proof of sufficiency in~\eqref{eq:reduction-balancedness}
  follows from~\eqref{eq:5} and \(\ri(\Psd{n} \cap \set{X}^{\perp}) =
  \setst{Z \in \Psd{n}}{\Image(Z) = \Null(X)}\).  For the proof of
  necessity, recall~\eqref{eq:5} and suppose that there exists \(X \in
  \Elliptope{n}\) such that \(-\oprodsym{b} = \Diag(y) - S\) for some
  \(y \in \Reals^n\) and \(S \in \ri(\Psd{n} \cap \set{X}^{\perp})\).
  Then \(0 = -\oprodsym{b}\ones = (\Diag(y)-S)\ones=y-S\ones\) shows
  that
  \begin{equation}
    \label{eq:47}
    y=S\ones.
  \end{equation}
  Since \(X\) and \(\oprodsym{\ones}\) are optimal solutions
  for~\eqref{eq:maxcut-sdp}, we find that \(0 =
  \Tr(-\oprodsym{b}\oprodsym{\ones}) = \Tr(-\oprodsym{b}X) =
  \iprodt{y}{\diag(X)}-\Tr(SX)\) so \(\iprodt{\ones}{y} = \Tr(SX) =
  0\).  By~\eqref{eq:47}, \(\qform{S}{\ones} = \iprodt{\ones}{y} =
  0\), so \(\ones \in \Null(S)\) and \(y = 0\).
\end{proof}

We are now in position to present the main result of this section:
\begin{theorem}
  \label{thm:rank-one}
  Let \(n \geq 4\) be an integer.  Equip \(\Sym{n}\) with the norm
  \(\pnorm[\infty]{\matvec(\cdot)}\).  Set
  \begin{gather*}
    \Omega_M
    \coloneqq
    \setst{
      C \in \bd(\Normal{\Elliptope{n}}{\oprodsym{\ones}})
    }{
      C \preceq 0,\,
      \rank(C) = 1,\,
      \pnorm[\infty]{\matvec(C)} = 1
    }
    \subseteq \Sym{n},
    \\
    \Gscr_M
    \coloneqq
    \setst{
      C \in \Omega_M
    }{
      \text{strict complementarity holds for~\eqref{eq:maxcut-sdp} with~\(C\)}
    },
    \\
    d \coloneqq \dim_H(\Omega_M).
  \end{gather*}
  Let \(\Sigma_d\) be the \(\sigma\)-algebra of \(H_d\)-measurable
  subsets of~\(\Sym{n}\) and set \(\Sigma_M \coloneqq \setst{\Ascr_M
    \in \Sigma_d}{\Ascr_M \subseteq \Omega_M}\).  Then
  \begin{enumerate}[(i)]
    \setlength{\itemsep}{4pt}
  \item \(\Omega_M \in \Sigma_d\) and \(\Gscr_M \in \Sigma_M\),
  \item \(H_{n-2}(\Omega_M) \in (0,\infty)\), so \(d = n-2\),
  \item \(H_d(\Gscr_M) > 0\) and \(H_d(\overline{\Gscr_M}) > 0\),
    where \(\overline{\Gscr_M} \coloneqq \Omega_M \setminus \Gscr_M\).
  \end{enumerate}
  In~particular, if we set
  \begin{equation}
    \label{eq:48}
    \Prob_M(\Ascr_M)
    \coloneqq
    \frac{H_d(\Ascr_M)}{H_d(\Omega_M)}
    \qquad
    \forall \Ascr_M \in \Sigma_M,
  \end{equation}
  then \((\Omega_M,\Sigma_M,\Prob_M)\) is a probability space and the
  event \(\Gscr_M\) satisfies \(\Prob_M(\Gscr_M) \in (0,1)\).
\end{theorem}

\begin{proof}
  \newcommand*{\Bbal}{\Bscr_{\textrm{bal}}}
  \newcommand*{\BbalmU}{\Bscr_{\textrm{bal},m,U}}
  \newcommand*{\GM}{\Gscr_M}
  \newcommand*{\GV}{\Gscr_V}
  \newcommand*{\Zi}{\Zscr_i}
  \newcommand*{\Znot}{\Zscr_{\emptyset}}
  We start by proving that
  \begin{equation}
    \label{eq:49}
    \Omega_M \in \Sigma_d.
  \end{equation}
  By standard Hausdorff measure theory, \(\Sigma_d\) contains every
  Borel set of~\(\Sym{n}\); see, e.g., \cite[Theorem~27]{Rogers98a}.
  Recall that the Borel sets of~\(\Sym{n}\) are the elements of the
  smallest \(\sigma\)-algebra on~\(\Sym{n}\) that contains all the
  open subsets of~\(\Sym{n}\).  For distinct \(i,j \in [n]\), set
  \(\Bscr_{ij} \coloneqq e_i - [0,1]e_j\).  For each \(S \in
  \tbinom{[n]}{3}\) and \(m \in \Naturals \setminus \set{0}\), define
  \begin{equation*}
    \Bscr_{S,m}
    \coloneqq
    \setst{
      b \in \Reals^n
    }{
      b \perp \ones,\,
      \pnorm[\infty]{b} = 1,\,
      \abs{b_i} \geq \tfrac{1}{m}\,\forall i \in S
    }.
  \end{equation*}
  Clearly, each \(\Bscr_{ij}\) and \(\Bscr_{S,m}\) is compact.  Let
  \(\varphi \colon b \in \Reals^n \mapsto -\oprodsym{b} \in \Sym{n}\).
  By \Cref{prop:omega},
  \begin{equation}
    \label{eq:50}
    \Omega_M
    =
    \bigcup_{i \in [n]}
    \bigcup_{j \in [n]\setminus\set{i}}
    \varphi(\Bscr_{ij})
    \cup
    \bigcup_{m=1}^{\infty}
    \bigcup_{S \in \tbinom{[n]}{3}}
    \varphi(\Bscr_{S,m}).
  \end{equation}
  Since each \(\varphi(\Bscr_{ij})\) and each \(\varphi(\Bscr_{S,m})\)
  is compact, \eqref{eq:50} shows that \(\Omega_M\) is an
  \(F_{\sigma}\), i.e.,~a~countable union of closed sets, and hence a
  Borel set.  This proves~\eqref{eq:49}.

  Next we prove that
  \begin{equation}
    \label{eq:51}
    H_{n-2}(\Omega_M) \in (0,\infty)
  \end{equation}
  from which it will follow via~\eqref{eq:18} that
  \begin{equation}
    \label{eq:52}
    d = n-2.
  \end{equation}
  Again we shall use \Cref{prop:omega}.  By \Cref{cor:bi-Lipschitz}
  and \Cref{thm:bi-Lipschitz-tensor},
  \begin{equation}
    \label{eq:53}
    H_1\paren[\Big]{
      \bigcup_{i \in [n]}
      \bigcup_{j \in [n]\setminus\set{i}}
      \varphi(\Bscr_{ij})
    }
    \in (0,\infty).
  \end{equation}
  Moreover,
  \begin{equation*}
    \Omega_M
    \supseteq
    \setst{
      -\oprodsym{b}
    }{
      b = -1 \oplus c,\,
      c \in \Reals_+^{n-1},\,
      \iprodt{\ones}{c} = 1
    }
    \implies
    H_{n-2}(\Omega_M) > 0.
  \end{equation*}
  For each \(s \in \set{\pm1}^n\) and \(i \in [n]\), the polytope
  \(\Bscr_{s,i} \coloneqq \setst{b \in \Reals_s^n}{b \perp
    \ones,\, -\ones\leq b\leq\ones,\, b_i = s_i}\) has dimension less
  than or equal to \(n-2\).  Since
  \begin{equation*}
    \Omega_M
    \subseteq
    \Nscr \cup
    \bigcup_{s \in \set{\pm1}^n} \bigcup_{i \in [n]} \varphi(\Bscr_{s,i})
  \end{equation*}
  for some set \(\Nscr\) of zero \(d\)-dimensional Hausdorff measure,
  and each \(\varphi(\Bscr_{s,i})\) has finite \(d\)-dimensional
  Hausdorff measure by \Cref{cor:bi-Lipschitz} and
  \Cref{thm:bi-Lipschitz-tensor}, the proof of~\eqref{eq:51} is
  complete.

  In the remainder of the proof we shall use subsets of~\(\Reals^n\)
  with constraints on the coordinates that are zero:
  \begin{equation*}
    \Zi
    \coloneqq
    \setst{b \in \Reals^n}{b_i = 0}
    \quad
    \forall i \in [n],
    \qquad\text{and}\qquad
    \Znot
    \coloneqq
    \Reals^n \setminus \bigcup_{i \in [n]} \Zi
    =
    \setst{b \in \Reals^n}{\supp(b) = [n]}.
  \end{equation*}
  Define also
  \begin{gather*}
    \Omega_V
    \coloneqq
    \setst{
      b \in \Reals^n
    }{
      b \perp \ones,\,
      \card{\supp(b)} \geq 3,\,
      \pnorm[\infty]{b} = 1
    },
    \\
    \GV
    \coloneqq
    \setst{
      b \in \Omega_V
    }{
      -\oprodsym{b} \in \Gscr_M
    },
    \\
    \overline{\GV}
    \coloneqq
    \Omega_V \setminus \GV,
    \\
    \Bbal
    \coloneqq
    \setst{
      b \in \Omega_V
    }{
      b \text{ is strictly balanced}
    },
    \\
    \overline{\Bbal}
    \coloneqq
    \Omega_V \setminus \Bbal.
  \end{gather*}
  \Cref{prop:sc-iff-sb} implies that
  \begin{gather}
    \label{eq:54}
    \GV \cap \Znot = \Bbal \cap \Znot,
    \\
    \overline{\GV} \cap \Znot = \overline{\Bbal} \cap \Znot.
  \end{gather}
  For each \(i \in [n]\), we have \(\GV \cap \Zi \subseteq
  \Omega_V \cap \Zi\) and the set on the RHS has zero
  \(d\)-dimensional Hausdorff measure.  Hence,
  \begin{equation}
    \label{eq:55}
    H_d(\GV \cap \Zi) = 0
    \qquad
    \forall i \in [n].
  \end{equation}
  Define \(\varphi_s\) as in~\eqref{eq:41} for each \(s \in
  \set{\pm1}^n\).  By putting together \eqref{eq:53}, \eqref{eq:55},
  and~\eqref{eq:54}, we find that
  \begin{equation}
    \label{eq:56}
    \GM
    =
    \Nscr \cup \bigcup_{s \in \set{\pm1}^n} \varphi_s(\GV \cap \Znot \cap \Reals_s^n)
    =
    \Nscr \cup \bigcup_{s \in \set{\pm1}^n} \varphi_s(\Bbal \cap \Znot \cap \Reals_s^n)
  \end{equation}
  for some subset \(\Nscr \subseteq \Omega_M\) such that \(H_d(\Nscr) = 0\).

  Let us prove that
  \begin{equation}
    \label{eq:57}
    \GM \in \Sigma_M.
  \end{equation}
  For each \(m \in \Naturals \setminus \set{0}\) and each \(U \in \tbinom{[n]}{3}\),
  define
  \begin{equation*}
    \BbalmU
    \coloneqq
    \setst[\Big]{
      b \in \Reals^n
    }{
      b \perp \ones,\,
      \pnorm[\infty]{b} = 1,\,
      \abs{b_i} \geq \tfrac{1}{m}\,\forall i \in U,\,
      \abs{b_i} + \tfrac{1}{m} \leq \sum_{j \in [n]\setminus\set{i}} \abs{b_j}
      \,\forall i \in [n]
    }.
  \end{equation*}
  Clearly, \(\Bbal = \bigcup_{m=1}^{\infty} \bigcup_{U \in \tbinom{[n]}{3}} \BbalmU\).  Hence,
  by~\eqref{eq:56},
  \begin{equation}
    \label{eq:58}
    \GM = \Nscr \cup \bigcup_{m=1}^{\infty} \bigcup_{U \in
      \tbinom{[n]}{3}} \bigcup_{s \in \set{\pm1}^n}
    \varphi_s(\BbalmU \cap \Znot \cap \Reals_s^n).
  \end{equation}
  Since each \(\varphi_s(\BbalmU \cap \Znot \cap \Reals_s^n)\) is
  compact, it follows that \(\GM\) is the union of a null set with
  an~\(F_{\sigma}\), and hence \(\GM \in \Sigma_d\).  This
  proves~\eqref{eq:57}.

  Set
  \begin{equation*}
    \xcirc \coloneqq 1 \oplus \frac{1}{n-1} \oplus \frac{-n}{(n-1)(n-2)} \ones
    \in \Reals^n,
    \quad
    \eps \coloneqq \frac{3}{4(n-1)(n-2)},
    \quad
    \text{and}
    \quad
    s(x) \coloneqq 1 \oplus 1 \oplus -\ones \in \set{\pm1}^n.
  \end{equation*}
  It is not hard to verify that
  \begin{equation}
    \label{eq:59}
    \xcirc + \eps(\Ball_{\infty} \cap \set{e_1,\ones}^{\perp})
    \subseteq
    \Bbal \cap \Znot \cap \Reals_{s(x)}^n.
  \end{equation}
  Since the set in the LHS of~\eqref{eq:59} has positive
  \(d\)-dimensional measure, so does the set in the RHS of~\eqref{eq:59},
  whence
  \begin{equation}
    \label{eq:60}
    H_d(\GM) > 0
  \end{equation}
  by \Cref{cor:bi-Lipschitz}, \Cref{thm:bi-Lipschitz-tensor},
  and~\eqref{eq:56}.

  Set
  \begin{equation*}
    \ycirc \coloneqq 1 \oplus -\frac{1}{n-1}\ones \in \Reals^n,
    \quad
    \delta \coloneqq \frac{1}{2(n-1)},
    \quad
    \text{and}
    \quad
    s(y) \coloneqq 1 \oplus -\ones \in \set{\pm1}^n.
  \end{equation*}
  It is not hard to verify that
  \begin{equation}
    \label{eq:61}
    \ycirc + \delta(\Ball_{\infty} \cap \set{e_1,\ones}^{\perp})
    \subseteq
    \overline{\Bbal} \cap \Znot \cap \Reals_{s(y)}^n.
  \end{equation}
  Hence,
  \begin{equation*}
    \overline{\GM}
    \supseteq
    \varphi(\overline{\GV} \cap \Znot)
    =
    \varphi(\overline{\Bbal} \cap \Znot)
    \supseteq
    \varphi_{s(y)}\paren{
      \overline{\Bbal} \cap \Znot \cap \Reals_{s(y)}^n
    }
    \supseteq
    \varphi_{s(y)}\paren{
      \ycirc + \delta(\Ball_{\infty} \cap \set{e_1,\ones}^{\perp})
    }.
  \end{equation*}
  Thus,
  \begin{equation*}
    H_d(\overline{\GM}) > 0
  \end{equation*}
  by \Cref{cor:bi-Lipschitz} and \Cref{thm:bi-Lipschitz-tensor}.
\end{proof}

\section{Conclusion}

We proved in \Cref{sec:generic-sc-failure} that the MaxCut
SDP~\eqref{eq:maxcut-sdp} has the worst possible behavior with respect
to strict complementarity when the objective function is in the
boundary of the normal cone of the elliptope at any of its vertices.
At a first glance, this may seem surprising since the MaxCut SDP is so
elementary and has so many favorable
properties.  However, as we explain next, from a properly chosen
viewpoint this bad behavior is not so surprising.

Consider, for instance, the convex set \(\Cscr \subseteq \Reals^2\) in
\Cref{fig:football}.  For concreteness, an explicit description
of~\(\Cscr\) is given by
\begin{equation}
  \label{eq:62}
  \Cscr
  \coloneqq
  \setst{x \in \Reals^2}{\norm{x}+\abs{x_1}\leq 1}
  =
  \setst[\big]{
    x \in \Reals^2
  }{
    \abs{x_1}\leq1/2,\,\abs{x_2} \leq \sqrt{1-2\abs{x_1}}\,
  },
\end{equation}
and it is not hard to show that \(\Cscr\) is the projection of the
feasible region of an SDP.  It is intuitive and simple to verify that
\(\ones\) lies in (the boundary of) the normal cone of~\(\Cscr\) at
its vertex~\(e_2\), but \(\ones\) is not in the relative interior of
any normal cone of~\(\Cscr\).  We can trace this phenomenon to the
smooth, nonpolyhedral boundary of~\(\Cscr\) around~\(e_2\).  It is
straightforward to extend this example to~\(\Reals^3\) by considering
the solid of revolution obtained by rotating~\(\Cscr\) around the
\(e_2\) axis, i.e., an American football.

\begin{figure}
  \centering
  \begin{tikzpicture}

    \draw[fill=gray!30!white]
    plot [scale=1,domain=-0.5:0,smooth,samples=200,variable=\x] ({\x},{sqrt(1+2*\x)}) --
    plot [scale=1,domain=0:-0.5,smooth,samples=200,variable=\x,rotate=180] ({\x},{-sqrt(1+2*\x)}) --
    plot [scale=1,domain=-0.5:0,smooth,samples=200,variable=\x,rotate=180] ({\x},{sqrt(1+2*\x)}) --
    plot [scale=1,domain=0:-0.5,smooth,samples=200,variable=\x] ({\x},{-sqrt(1+2*\x)});

    \shade[bottom color=gray!30, top color=white] (0,1) -- (1.2,2.2) -- (-1.2,2.2) -- cycle;

    \draw[->, help lines, dashed, font=\scriptsize] (-1.2,0) -- (1.2,0) node[right] {$x_1$};
    \draw[->, help lines, dashed, font=\scriptsize] (0,-1.2) -- (0,2.2) node[above] {$x_2$};
    \foreach \x in {-1,1} {
      \draw[help lines,font=\scriptsize] (\x,2pt)--(\x,-2pt) node[below] {$\x$};
    }

    \draw[scale=1,domain=-0.5:0,smooth,samples=200,variable=\x] plot ({\x},{sqrt(1+2*\x)});
    \draw[scale=1,domain=-0.5:0,smooth,samples=200,variable=\x,rotate=180] plot ({\x},{sqrt(1+2*\x)});
    \draw[scale=1,domain=-0.5:0,smooth,samples=200,variable=\x] plot ({\x},{-sqrt(1+2*\x)});
    \draw[scale=1,domain=-0.5:0,smooth,samples=200,variable=\x,rotate=180] plot ({\x},{-sqrt(1+2*\x)});

    \draw[->, dashed] (0,1) -- (1.2,2.2);
    \draw[->, dashed] (0,1) -- (-1.2,2.2);

    \draw[->, thick] (0,1) -- (0.5,1.5) node[below right] {\(\mathbbm{1}\)};

    \node at (0,-0.5) {\(\mathscr{C}\)};
    \node[font=\scriptsize] at (0,1.8) {\(\mathcal{N}(\mathscr{C};e_2)\)};

  \end{tikzpicture}
  \caption{The set \(\Cscr\) defined in~\eqref{eq:62} and its normal cone \(\Ncal(\Cscr;e_2)\) at \(e_2\).}
  \label{fig:football}
\end{figure}

The elliptope looks somewhat similar to~\(\Cscr\) in the following
sense.  Let us consider the projection \(\Elliptope{n}' \subseteq
\Reals^{\tbinom{n}{2}}\) of the elliptope \(\Elliptope{n}\) into its
off-diagonal entries.  For \(n \geq 3\), the set \(\Elliptope{n}'\) is
a compact nonpolyhedral convex set with \(2^{n-1}\) vertices by
\Cref{thm:vertices-elliptope}.  Intuitively, \(\Elliptope{n}'\) can be
thought of as being obtained from the polytope which is the convex
hull of these \(2^{n-1}\) vertices by inflating it like a balloon,
while preserving the vertices fixed.  (In~fact, by
\cite[Proposition~2.9]{LaurentP95a}, the line segments between the
\(2^{n-1}\) vertices are also kept fixed.)  In~this~way,
\(\Elliptope{n}'\) is a round, plump convex set, whose boundary is
smooth almost everywhere, and the neighborhood of \(\Elliptope{n}'\)
around any vertex looks like (a generalization of) what is depicted by
the set~\(\Cscr\) from the previous paragraph.  Thus, when one
considers that the elliptope around a vertex ``locally'' looks
like~\(\Cscr\) around~\(e_2\), the poor behavior of the MaxCut SDP
described in~\Cref{sec:generic-sc-failure} makes more intuitive sense.
The discussion above indicates a natural direction for future
research.  Namely, to extend \Cref{thm:generic-sc-failure} to more
general SDPs, by requiring the feasible region to be ``locally
nonpolyhedral'' around its vertices.

\end{document}